\documentclass{amsart}
\usepackage{amssymb,amsfonts, color,epsf}
\usepackage [cmtip,arrow]{xy}
\xyoption{all}
\usepackage {pb-diagram,pb-xy}
\usepackage{enumerate}
\usepackage{accents}
\usepackage[noautoscale]{youngtab}

\ifx\pdfpageheight\undefined
   \usepackage[dvips,colorlinks=true,linkcolor=blue,citecolor=red,%
      urlcolor=green]{hyperref}
   \usepackage[dvips]{graphicx}
   \makeatletter
   \edef\Gin@extensions{\Gin@extensions,.mps}
=
   \makeatother
\else

 \usepackage[pdftex]{graphicx}
   \usepackage[bookmarksopen=false,pdftex=true,breaklinks=true,%
      backref=page,pagebackref=true,plainpages=false,%
      hyperindex=true,pdfstartview=FitH,colorlinks=true,%
      pdfpagelabels=true,colorlinks=true,linkcolor=blue,%
      citecolor=red,urlcolor=green,hypertexnames=false%
      ]%
   {hyperref}
\fi
\usepackage{bm}

\usepackage{amsmath,amssymb,amsfonts}
\numberwithin{equation}{subsection}

\newtheorem{theorem}{Theorem}
\newtheorem{theorem-appendix}[subsection]{Theorem}

\newtheorem{lemma}[equation]{Lemma}

\newtheorem{corollary}{Corollary}
\newtheorem{corollary-appendix}[subsection]{Corollary}

\newtheorem{proposition}[equation]{Proposition}

\theoremstyle{definition}
\newtheorem{definition}[equation]{Definition}

\newtheorem{notation}[equation]{Notation}

\theoremstyle{remark}
\newtheorem{remark}[equation]{Remark}

\theoremstyle{observation}
\newtheorem{observation}[equation]{Observation}

\theoremstyle{properties}
\newtheorem{properties}[equation]{Properties}

\newtheorem{innercustomthm}{Claim}
\newenvironment{customclaim}[1]
  {\renewcommand\theinnercustomthm{#1}\innercustomthm}
  {\endinnercustomthm}

\let\oldsection\section
\renewcommand{\section}{
  \renewcommand{\theequation}{\thesection.\arabic{equation}}
  \oldsection}
\let\oldsubsection\subsection
\renewcommand{\subsection}{
  \renewcommand{\theequation}{\thesubsection.\arabic{equation}}
  \oldsubsection}

\newcommand{\hide}[1]{}

\newcommand{\bbA}{{\mathbb A}}

\newcommand{\bbF}{{\mathbb F}}

\newcommand{\bbN}{{\mathbb N}}

\newcommand{\bbP}{{\mathbb P}}
\newcommand{\bbQ}{{\mathbb Q}}
\newcommand{\bbR}{{\mathbb R}}

\newcommand{\bbZ}{{\mathbb Z}}

\newcommand{\bF}{{\bf F}}

\newcommand{\cL}{{\mathcal L}}

\newcommand{\cP}{{\mathcal P}}

\newcommand{\rH}{{\rm H}}

\newcommand*{\bigchi}{\mbox{\huge$\chi$}}

\newcommand{\Hom}{{\rm Hom}}

\newcommand{\Spec}{{\rm Spec}}

\DeclareMathOperator{\HH}{H}
\DeclareMathOperator{\RR}{R}

\newcommand{\eps}{\varepsilon}

\newcommand{\Reali}{\mathcal{R}}

\newcommand{\card}{\mathrm{card}}
\newcommand{\vcd}{\mathrm{vcd}}

\newcommand{\Cube}{\mathrm{Cube}}
\newcommand{\Tube}{\mathrm{Tube}}
\newcommand{\CTube}{\mathrm{TubeCompl}}
\newcommand{\Boundary}{\mathrm{TubeBoundary}}

\begin{document}
\title{VC density of definable families over valued fields}
\author{Saugata Basu}
\address{Department of Mathematics, Purdue University,
150 N. University Street, West Lafayette, IN 47907, U.S.A.}
\email{sbasu@math.purdue.edu}
\author{Deepam Patel}
\address{Department of Mathematics, Purdue University,
150 N. University Street, West Lafayette, IN 47907, U.S.A.}
\email{patel471@purdue.edu}

\thanks{
S.B. would like to acknowledge support from the National Science Foundation awards and DMS-1620271 and CCF-1618918. 
D.P. would like to acknowledge support from the National Science Foundation award DMS-1502296.
}

\begin{abstract}
We prove a tight bound on the number of realized $0/1$ patterns 
(or equivalently on the Vapnik-Chervonenkis codensity) of definable families in models of the theory of algebraically closed valued fields with a non-archimedean valuation.
Our result improves the best known result in this direction proved by
Aschenbrenner, Dolich, Haskell, Macpherson and Starchenko,  who proved a weaker bound in the restricted case where the characteristics of the field $K$ and its residue field are both assumed to be $0$. The bound obtained here is optimal and without any restriction on the characteristics.

We obtain the aforementioned bound as a consequence of another result on bounding the Betti numbers of semi-algebraic subsets of certain Berkovich analytic spaces, mirroring similar results known already in the case of o-minimal structures and for real closed, as well as, 
algebraically closed fields. The latter result is the first result in this direction and is possibly of independent interest. 
Its proof relies heavily on recent results of Hrushovski and Loeser 
on the topology of semi-algebraic subsets of Berkovich analytic spaces.

\end{abstract}
\maketitle

\tableofcontents

\section{Introduction}

In this article, we prove a tight bound on the number of realized $0/1$ patterns 
(or equivalently on the Vapnik-Chervonenkis codensity) of definable families in models of the theory of algebraically closed valued fields with a non-archimedean valuation (henceforth referred to just as ACVF). This result improves on the best known upper bound on this quantity previously obtained by Aschenbrenner  et al. in \cite{Starchenko-et-al}. Our result is a consequence
of a topological result giving an upper bound on the Betti numbers of certain semi-algebraic sets obtained
as Berkovich analytifications of definable sets in certain models of ACVF which we will recall more precisely in the
next section.\\
   
In order to state our main combinatorial result we need to introduce some preliminary notation and definitions.

\subsection{Combinatorial definitions}

\label{subsec:0-1}
Suppose $V$ and $W$ are sets, and $X \subset V \times W$ is a subset.
Let $\pi_V: X \rightarrow V, \pi_W: X \rightarrow W$ denote the restriction to $X$ of the natural projection maps. For any $v \in V, w \in W$, we set 
$X_v :=\pi_W(\pi_V^{-1}(v))$, and
$X_w :=\pi_V(\pi_W^{-1}(w))$.

\begin{notation}
\label{not:chi}
For each $n>0$, we define a function
\[
\bigchi_{X,V,W;n} : V \times W^n \rightarrow \{0,1\}^{n}
\]
as follows.
For 
$\bar{w} := (w_1,\ldots,w_n) \in W^{n}$ and $v \in V$, 
we set
\begin{equation*}
(\bigchi_{X,V,W;n} (v,\bar{w}))_i :=  
\begin{cases}
0 \mbox{ if $ v \notin X_{w_i}$} \\
1 \mbox{ otherwise.} 
\end{cases}
\end{equation*}
(Note that in the special case when $n=1$, $\bigchi_{X,V,W;1}$ is just the usual characteristic
function of the subset $X \subset V \times W$).

For $\bar{w} \in W^n$ and $\sigma \in \{0,1\}^n$, we will say that $\sigma$ is \emph{realized  by the tuple $(X_{w_1},\ldots,X_{w_n})$ of subsets of  $V$} if there exists $v \in V$ such that
 $\bigchi_{X,V,W;n} (v,\bar{w}) =  \sigma$. We will often refer to elements of $\{0,1\}^n$ colloquially as `$0/1$ patterns'. 
 \end{notation}
 
Finally, we define the function
 \[
 \bigchi_{X,V,W}: \bbN \rightarrow \bbN
 \] 
 by
\begin{equation*}
\bigchi_{X,V,W}(n) := \max_{\bar{w} \in W^{n}} \card(\bigchi_{X,V,W;n}(V,\bar{w})).
\end{equation*}

The function $\bigchi_{X,V,W}$ is closely related to the notion of \emph{VC-codensity of a set system}. 
Since some of the
prior results (for example, those in \cite{Starchenko-et-al})
have been stated in terms of VC-codensity it is useful to recall its definition here.

\begin{definition}
\label{def:vcd}
Let $X$ be a set and $\mathcal{S} \subset 2^X$. The \emph{shatter function of $\mathcal{S}$}, $\pi_{\mathcal{S}}:\bbN \rightarrow \bbN$,  is defined by setting
\[
\pi_{\mathcal{S}}(n) := \max_{A \subset X, \card(A) = n} \card(\{ A \cap Y \mid Y \in \mathcal{S}\}).
\]
We denote
\[
\vcd_{\mathcal{S}} := \limsup_{n \rightarrow \infty} \frac{\log(\pi_{\mathcal{S}}(n))}{\log(n)}.
\]
\end{definition}

Given a definable subset $X \subset V \times W$ in some structure,
we will denote 
\[
\vcd(X,V,W) := \vcd_{\mathcal{S}},
\] 
where $\mathcal{S} = \{X_v | v \in V\} \subset 2^W$.
We will call (following the convention in \cite{Starchenko-et-al}), $\vcd(X,V,W)$,   
the \emph{VC-codensity}  of the family of subsets,  $\{X_w | w \in W\}$, of $V$. More generally, if $\phi(\overline{X},\overline{Y})$ is a first-order formula (with parameters) in the theory of some structure $M$, we set
\[
\vcd(\phi) := \vcd(S,M^{|\bar{X}|},M^{|\bar{Y}|}),
\]  
where $S \subset M^{|\bar{X}|} \times   M^{|\bar{Y}|}$ is the set defined by $\phi$.
(Here and elsewhere in the paper, $|\overline{X}|$ denotes the length of the finite tuple of variables
$\overline{X}$.)
Note also that if $M$ is an NIP structure (see for example \cite[Chapter 2]{Simon} for definition), then 
$\vcd(\phi) < \infty$ for every (parted) formula $\phi$. \\

The problem of proving upper bounds on $\vcd(X,V,W)$ of a definable family can be reduced to proving
upper bounds on the function $\bigchi_{X,V,W}$ (see Proposition~\ref{prop:vcd} below). We will henceforth
concentrate on the problem of obtaining tight upper bounds on the function $\bigchi_{X,V,W}$ for the rest of the paper.

\subsection{Brief History}
For definable families of hypersurfaces  in $\bbF^k$ of fixed degree over a field $\bbF$, 
Babai, Ronyai, and Ganapathy \cite{RBG01} gave an elegant argument using linear algebra to show that the number of  $0/1$ patterns (cf. Notation \ref{not:chi})  
realized by $n$ such hypersurfaces in $\bbF^k$ is bounded by $C \cdot n^k$, where $C$ is a constant that depends on the family (but independent
of $n$).
 This bound is easily seen to be optimal. A more refined topological estimate on these realized $0/1$ patterns (in terms of the sums of the Betti numbers) is given in 
\cite{BPR-tight}, where the methods are more in line with the methods in the current paper.\\
 
A similar result was proved in \cite{BPR8}  for definable families of semi-algebraic sets in $\mathrm{R}^k$,  where $\mathrm{R}$ is an arbitrary real closed field.
For definable families in $M^{k}$, where $M$ is an arbitrary o-minimal expansion of a real closed field, the first author \cite{Basu9} adapted the methods in
\cite{BPR8}  to prove a bound of $C \cdot n^k$ on the number of $0/1$ patterns
for such families where $C$ is a constant that depends on the family (see also \cite{JL}). 
These bounds were obtained as a consequence
of more general results bounding the individual Betti numbers of definable sets defined in terms of the members
of the family, and more sophisticated homological techniques (as opposed to just linear algebra) played an important role in obtaining these bounds.\\

If $K$ is an algebraically closed valued field, then the problem of obtaining tight bounds 
on $\vcd(\phi)$ for parted formulas, $\phi(\overline{X},\overline{Y})$, in the 
one sorted language of valued fields with parameters in $K$
was considered by  Aschenbrenner  et al. in \cite{Starchenko-et-al}. In particular, they obtained
the nontrivial bound of $2|\overline{X}|$ on $\vcd(\phi)$
in the case when the characteristic pair of $K$ (i.e. the pair consisting of the characteristic of the field $K$ and that of its residue field) is $(0,0)$ \cite[Corollary 6.3]{Starchenko-et-al}.
\\

Given that the model-theoretic/algebraic techniques used thus far do not immediately yield the tight upper bound of 
$|\overline{X}|$ on $\vcd(\phi(\overline{X},\overline{Y}))$ for valued fields, it is natural to consider a more topological approach as in \cite{Basu9}. However, for definable families over a (complete) valued field, it is not a priori clear
that there exists an appropriate well-behaved cohomology theory (i.e. with the required finiteness/cohomological dimension properties) that makes the approach in
\cite{Basu9} feasible in this situation. 
For example, ordinary sheaf cohomology with respect to the Zariski or \'{E}tale site for schemes are clearly not suitable.
Fortunately, the recent break-through results of
Hrushovski  and Loeser \cite{HL} give us an opening in this direction.  Instead of considering the original definable subset of an affine variety $V$ defined over $K$, we can consider
the corresponding \emph{semi-algebraic} subset of the Berkovich analytification 
$B_\bF(V)$
of $V$ (see \S \ref{subsec:HL} below for the definitions). These semi-algebraic subsets have certain key topological tameness properties which are analogous to those used in the case of o-minimal structures, and moreover crucially they are homotopy equivalent to a simplicial complex of dimension at most $\dim(V)$. Therefore, their cohomological dimension is at most $\dim(V)$. In particular, the singular cohomology of the underlying topological spaces satisfies the requisite properties. 
Thus, in order to bound the number of realizable $0/1$ patterns of a finite set of definable subsets of $V$,
we can  first replace the finite set of definable subsets of $V$ by the corresponding semi-algebraic
subsets of $B_\bF(V)$, and then try to make use of their tame topological properties to obtain a bound 
on the number of $0/1$ patterns realized by these semi-algebraic subsets. An upper bound on the latter quantity will also be an upper bound on the number of $0/1$ patterns realized by the definable subsets that we started with (this fact is elucidated later in Observation~\ref{obs:extension} in \S~\ref{subsec:proof-of-0-1}). \\

Using the results of  Hrushovski and Loeser, one can then hope to proceed with the o-minimal case as the guiding principle. While
the arguments are somewhat similar in spirit, there are several technical challenges that need to be overcome -- for example, 
an appropriate definition of ``tubular neighborhoods" with the required properties (see \S \ref{subsec:outline} below for a more detailed description of these challenges). 
The bounds on the
sum of the Betti numbers of the semi-algebraic subsets of Berkovich spaces that we obtain in this way are exactly analogous to the ones
in the algebraic, semi-algebraic,  as well as o-minimal cases. The fact that the cohomological dimension of the semi-algebraic subsets of 
$B_\bF(V)$ is bounded by $\dim(V)$ is one key ingredient in obtaining these tight bounds.  \\

Our results on bounding the Betti numbers of semi-algebraic subsets of Berkovich spaces are of independent interest, and the aforementioned results seem to suggest a more general formalism of cohomology associated to NIP structures. 
For example, one obtains bounds  (on the Betti numbers)  of the exact same shape and having the same exponents for definable families in the case of 
algebraic, semi-algebraic, o-minimal  and  valued field structures. Moreover, in each of these cases, these bounds are obtained as a consequence of general bounds on the dimension of certain cohomology groups. Therefore, it is perhaps reasonable to hope for some general cohomology theory (say for NIP structures which are fields)  which would in turn give a uniform method of obtaining tight bounds on VC-density via cohomological methods. More generally, it shows that cohomological methods can play an important role in model theory in general.  \\

As a consequence of the bound on the Betti numbers (in fact using the bound only on the $0$-th Betti number) 
we prove that
$\vcd(\phi(\overline{X},\overline{Y}))$
over an arbitrary algebraically closed valued field is  bounded by 
$|\overline{X}|$. 
One consequence of our methods (unlike the techniques used in \cite{Starchenko-et-al}) is that there are no restrictions on the characteristic pair of the field $K$. \\

Finally note that in \cite{Starchenko-et-al}  the authors also obtain a bound of 
$2|\overline{X}| - 1$
on $\vcd(\phi(\overline{X},\overline{Y}))$,
over $\bbQ_p$, 
where $\phi$ is a formula in Macintyre's language \cite{Macintyre}. 
However, our methods right now do not yield results in this case. \\

{\bf Outline of the paper:} In \S\ref{sec:main} we first  introduce the necessary  technical background (in \S\ref{subsec:acvf}),  and then state the main results  of the paper, namely Theorems~\ref{thm:0-1} and \ref{thm:main}, and Corollary~\ref{cor:vcd-acvf} (in \S\ref{subsec:new}). The proofs of the main results  appear in \S\ref{sec:proofs}.
We first give an outline of the proofs in \S  \ref{subsec:outline}.
We next prove the main topological result of the paper (Theorem~\ref{thm:main}) in  
in \S \ref{subsec:proof-of-main},  and prove
Theorem~\ref{thm:0-1} and Corollary~\ref{cor:vcd-acvf} in \S\ref{subsec:proof-of-0-1} and
\S\ref{subsec:proof-of-vcd-acvf} respectively. \\

In order to make the paper self-contained and for the benefit of the readers, 
we include in  an appendix (Appendix \S\ref{sec:appendix}) a review of some
very classical results about singular cohomology (in  \S\ref{subsec:cohomology}), as well as  much more recent ones related to semi-algebraic sets associated
to definable sets in models of ACVF proved by Hrushovski and Loeser \cite{HL}  (in \S\ref{subsec:HL}).
These results are used heavily in the proofs of the main theorems.

\section{Main results}
\label{sec:main}

\subsection{Model theory of algebraically closed valued fields}
\label{subsec:acvf}

In this section
$K$ will always denote an algebraically closed non-archimedean valued field $K$, and the value group of $K$ will be denoted by $\Gamma$.
Let $R:= K[X_1,\ldots,X_N]$ and $\bbA_K^N = \Spec(R)$. Given a closed affine subvariety 
$V  = \Spec(A)$ of  $\bbA_K^N = \Spec(R)$ and an extension $K'$ of $K$, we will denote
by $V(K') \subset \bbA_K^N(K')$ the set of $K'$ points of $V$. \\

We denote by $\cL$ the two-sorted language 
\[
(0_K,1_K,+_K,\times_K, | \cdot |:K \rightarrow \Gamma \cup \{0_\Gamma\}, \leq_{\Gamma}, \times_{\Gamma}),
\]
where the subscript $K$ denotes constants, functions, relations etc., of the field sort and the subscript $\Gamma$ denotes the same for the value group sort. When the context is clear we will drop the subscripts.
The constant $0_\Gamma$  is interpreted as the valuation of $0$ (and does not technically belong to the 
value group). \\

Now suppose that $\phi$ is a quantifier-free formula in the language $\cL(K;\Gamma \cup \{0_\Gamma\})$,
with free variables of only the field sort. Then, $\phi$ is a quantifier-free formula
with atoms of the form $|F| \leq \lambda \cdot |G|$ where $F,G \in R$ and $\lambda \in \Gamma \cup \{0_\Gamma\}$. 
The formula  $\phi$ gives rise to a definable subset of $\bbA^{N}_K$ and, in particular, $\phi$ defines a subset of $\bbA_K^N(K')$ for every valued extension $K'$ of $K$.
We will denote the intersection of this subset with $V$ by 
$\Reali(\phi,V)$, and by $\Reali(\phi,V)(K')$ the corresponding subset of $V(K')$.

\subsection{New Results}
\label{subsec:new}
Our main result is the following.

\begin{theorem}[Bound on  the number of $0/1$ patterns]
\label{thm:0-1}
Let $K$ be an algebraically closed valued field with value group 
$\Gamma$.
Suppose that $V \subset \bbA_K^N$ and $W \subset \bbA_K^M$ are closed affine subvarieties and let
\[
\phi(X_1,\ldots,X_N;Y_1,\ldots,Y_M)
\] 
be a formula with parameters in 
$(K;\Gamma \cup \{0_\Gamma\})$ in the  language $\cL$ (with free variables only of the field sort).
Then
there exists a constant $C= C_{\phi,V,W}$, such that for all $n > 0$,
\[
\bigchi_{\Reali(\phi, (V \times W))(K),V(K),W(K)}(n) \leq C \cdot n^{k},
\]
where $k = \dim V$.
\end{theorem}

As an immediate corollary  of  Theorem~\ref{thm:0-1} we obtain the following bound on the VC-codensity
for definable families over algebraically closed valued fields.

\begin{corollary}[Bound on the
VC-codensity
for definable families over ACVF]
\label{cor:vcd-acvf}
Let $K$ be an algebraically closed valued field with value group $\Gamma$.
Let $\phi(\overline{X},\overline{Y})$ be a formula with parameters in $(K;\Gamma \cup \{0_\Gamma\})$
in the  language $\cL$. Then,
\[
\vcd(\phi) \leq 
|\overline{X}|.
\]
\end{corollary}

Theorem~\ref{thm:0-1} will follow from a more general topological theorem which
we will now state. Before we state the theorem, we recall some more notation. \\

We now assume that $K$ is an algebraically closed complete valued field with a non-archimedean valuation whose value group
$\Gamma$  is a subgroup of the multiplicative group $\mathbb{R}_{>0}$.\\

Given an affine variety $V$ as before, Hrushovski-Loeser \cite{HL} associate to $V$ a locally compact Hausdorff topological space, denoted by $B_\bF(V)$. More generally, they associate a locally compact Hausdorff  topological space $B_{\bF}(X)$ to any definable subset $X \subset V$ which is functorial in definable maps. In the the present setting, $B_{\bF}(V)$ can be identified with the Berkovich analytic space associated to $V$ and has an explicit description in terms of valuations. We refer the reader to Appendix \ref{subsec:HL} for a brief review of this construction and its main properties.\\

Let $\phi$ be a formula in the language $\cL(K;\Gamma \cup \{0_\Gamma\})$ 
with free variables only of the field sort. 
Note that every
$\cL(K;\Gamma \cup \{0_\Gamma\})$-formula
is equivalent modulo the two-sorted theory of $(K;\Gamma \cup \{0_\Gamma\})$ to a quantifier-free formula (see for example \cite[Theorem~7.1 (ii)]{Haskell-et-al}).
Because of this fact, 
we can assume without loss of generality in what follows that $\phi$ is a quantifier-free formula, 
and is thus a quantifier-free formula
with atoms of the form $|F| \leq \lambda  \cdot |G|$ where $F,G \in R$ and $\lambda \in \Gamma \cup \{0_\Gamma\}$.  \\
 
\begin{notation}
\label{not:tilde-reali}
If $V \subset \bbA_K^N$ is a affine closed subvariety, and $\phi$ a formula in the language $\cL(K;\Gamma \cup \{0_\Gamma\})$ 
with free variables only of the field sort,  
we will denote
$\widetilde{\Reali}(\phi,V)$ the {\it semi-algebraic} subset $B_\bF(\Reali(\phi,V))$ of $B_\bF(V)$. \\
\end{notation}

Suppose now that $V \subset \bbA_K^N$ and $W \subset \bbA_K^M$ are closed affine subvarieties and let
$\phi(\cdot;\cdot)$ be a formula 
in disjunctive normal form without negations and with atoms of the form $|F| \leq \lambda \cdot |G|, F,G \in K[X_1,\ldots,X_N,Y_1,\ldots,Y_M], \lambda \in \Gamma \cup \{0_\Gamma\}$. Then for each $w \in W(K), \widetilde{\Reali}(\phi(\cdot,w),V)$ is a semi-algebraic subset
of $B_{\bF}(V)$.\\

For $\bar{w} = (w_1,\ldots,w_n) \in
W(K)^n$ and 
$\sigma \in \{0,1\}^n$, we set
\begin{equation}\label{eqn:realisigmabarw}
\widetilde{\Reali}(\sigma,\bar{w}) := 
\widetilde{\Reali}(\phi_\sigma(\bar{w}),V),
\end{equation}
where 
\[
\phi_\sigma(\bar{w}) := \bigwedge_{i, \sigma(i) =1} \phi(\cdot,w_i) \wedge \bigwedge_{i, \sigma(i) = 0} \neg\phi(\cdot,w_i).
\]

Given a topological space $Z$, we denote by $\HH^i(Z)$ the corresponding $i$-th singular cohomology group of $X$ with rational coefficients. We refer the reader to \S~\ref{subsec:cohomology} for a brief recollection of the main properties of these cohomology groups. We note that for $Z = \widetilde{\Reali}(\sigma,\bar{w})$ these cohomology groups are finite dimensional $\bbQ$-vector spaces. 
Let 
\[
b_i(\widetilde{\Reali}(\sigma,\bar{w})) = \dim_\bbQ \HH^i(\widetilde{\Reali}(\sigma,\bar{w}))
\] 
denote the 
corresponding $i$-th Betti number. \\

The following theorem, mirroring a similar theorem in the o-minimal case
\cite{Basu9},  is the main technical result of this paper.

\begin{theorem}[Bound on the Betti numbers]
\label{thm:main}
Let $K$ be an algebraically closed complete valued field with a non-archimedean valuation whose value group
$\Gamma$  is a subgroup of the multiplicative group $\mathbb{R}_{>0}$.
Suppose that $V \subset \bbA_K^N$ and $W \subset \bbA_K^M$ 
are closed affine subvarieties and let
$\phi(\cdot;\cdot)$ be a formula 
in disjunctive normal form without negations and with atoms of the form $|F| \leq \lambda \cdot |G|, F,G \in K[X_1,\ldots,X_N,Y_1,\ldots,Y_M], \lambda \in \Gamma \cup \{0_\Gamma\}$. 
Let $\dim(V)=k$. 
Then, there exists a constant $C = C_{\phi,V,W} > 0$ such that for all 
$\bar{w} \in W(K)^n$, 
and $0 \leq i \leq k$,
\[
\sum_{\sigma \in \{0,1\}^n} b_i(\widetilde{\Reali}(\sigma,\bar{w})) \leq C  n^{k - i}.
\]
\end{theorem}

\section{Proofs of the main results}
\label{sec:proofs}
In this section we prove our main results. Before starting the formal proof we first give a brief outline of our methods.
\subsection{Outline of the methods used to prove the main theorems}
\label{subsec:outline}
Our main technical result Theorem \ref{thm:main} gives a bound, for each $i, 0 \leq i \leq k$, and 
$\bar{w} \in W(K)^n$, 
on the sum over $\sigma\in \{0,1\}^n$ of the $i$-th Betti numbers
of $\widetilde{\Reali}(\sigma,\bar{w})$. The technique for achieving this is an adaptation of the topological methods used to prove a similar result in the
o-minimal category in \cite{Basu9} (Theorem 2.1). 
We recall here the main steps of the proof of Theorem 2.1 in \cite{Basu9}.\\

We assume that $V = \mathrm{R}^N, W = \mathrm{R}^M$, where $\mathrm{R}$ is a real closed field and $X \subset V \times W$ is a closed definable subset
in an o-minimal expansion of $\mathrm{R}$. 

\begin{enumerate}[Step 1.]
\item
\label{itemlabel:steps:1}
The first step in the proof  is to construct definable infinitesimal tubes around the fibers $X_{w_1},\ldots,X_{w_n}$.
\item
\label{itemlabel:steps:2}
Let $\sigma \in \{0,1\}^n$, and $C$ be a  connected component of 
$$\bigcap_{\sigma(i)  =1} X_{w_i} \cap \bigcap_{\sigma(i) = 0} (V \setminus X_{w_i}).$$ One proves that there exists
a unique connected component $D$ of the complement of the boundaries of the  tubes constructed in Step \ref{itemlabel:steps:1}  such that
$C$ is homotopy equivalent to $D$.
The homotopy equivalence is proved using the local conical structure theorem for o-minimal structures. 
\item
\label{itemlabel:steps:3}
As a consequence of Step \ref{itemlabel:steps:2}, in order to bound
$\sum_\sigma b_i(\RR(\sigma,\bar{w}))$, it suffices 
(using Alexander duality)  to bound the Betti numbers of the union of the boundaries of the tubes constructed in 
Step \ref{itemlabel:steps:1}.
\item
\label{itemlabel:steps:4}
Bounding the Betti numbers of the union of the boundaries of the tubes is achieved using
certain inequalities which follow from the Mayer-Vietoris exact sequence. In these inequalities only the Betti numbers of at most $k$-ary  
intersections of  the boundaries play a  role.
\item
\label{itemlabel:steps:5}
One then uses Hardt's triviality theorem for o-minimal structures to get a uniform bound on each of these Betti numbers that depends only on the
definable family under consideration i.e. on $X,V$, and $W$. Thus, the only part of the bound that grows with $n$ comes from certain binomial coefficients counting the number of different possible intersections one needs to consider.
\end{enumerate}

The method we use for proving Theorem \ref{thm:main} is close in spirit to the proof  of Theorem 2.1 in \cite{Basu9} as outlined above but different in many important details. For each of the steps enumerated above we list the corresponding step in the proof of  Theorem \ref{thm:main}.
\begin{enumerate}[Step 1$'$.]
\item
\label{itemlabel:steps:1'}
We construct again certain tubes around the fibers and give explicit descriptions of the tubes in terms of the formula $\phi$  
defining the given semi-algebraic set $\widetilde{\Reali}(\sigma,\bar{w})$. 
The definition of these tubes is somewhat more complicated than in the o-minimal case (see Notation~\ref{not:tubes}).
The use of two different infinitesimals to define these tubes is necessitated by the singular behavior of the semi-algebraic set defined by $|F| \leq \lambda |G|$ near the common zeros of $F$ and $G$.

\item
\label{itemlabel:steps:2'}
The  homotopy equivalence property  analogous to Step  \ref{itemlabel:steps:2} above is proved in Proposition \ref{prop:main}, and the role of local conical structure theorem in the o-minimal case  is now played by a corresponding result of Hrushovski and Loeser (see Theorem \ref{prop:sublevel} below).

\item
\label{itemlabel:steps:3'}
We avoid the use of Alexander duality by directly using
a Mayer-Vietoris type inequality giving a bound on the Betti numbers of intersections of open sets in terms of the Betti numbers of up to $k$-fold unions (cf. Proposition \ref{prop:spectral-inequality}). 

\item
\label{itemlabel:steps:4'}
This step is subsumed by Step \ref{itemlabel:steps:3'}$'$.

\item
\label{itemlabel:steps:5'}
Finally, instead of using Hardt's triviality to obtain a constant bound on the Betti numbers of these `small' unions, we 
use a theorem of Hrushovski and Loeser 
which states that the number of homotopy types amongst the fibers of any fixed map in the analytic category
that we consider is finite (cf. Theorem \ref{prop:finite-homotopy-types} below).
\end{enumerate}

 We apply Theorem \ref{thm:main} directly to obtain the VC-codensity
 bound in the case of the theory of ACVF (using Observation \ref{obs:extension}). 
One extra subtlety 
here is in removing the 
assumption on the formula $\phi$ (which occurs in the hypothesis of Theorem \ref{thm:main}).
Actually, in order to prove Corollary~\ref{cor:vcd-acvf} in general it suffices only to consider $\phi$ of the special form having just one atom of the form $|F| \leq \lambda \cdot |G|$ or $|F| = \lambda \cdot  |G|$. This reduction 
from the general case to the special case is encapsulated in a combinatorial result (Proposition \ref{prop:vcd'}). With the help of Proposition \ref{prop:vcd'}, Corollary~\ref{cor:vcd-acvf} becomes a consequence of Theorem~\ref{thm:main} and Observation~\ref{obs:extension}.\\

We now give the proofs in full detail. In the next subsection (\S \ref{subsec:proof-of-main}) we
give the proof of Theorem \ref{thm:main}. In \S  \ref{subsec:proof-of-0-1}, we show how to deduce Theorem~\ref{thm:0-1} from Theorem \ref{thm:main}. Finally, in \S\ref{subsec:proof-of-vcd-acvf} we show how to deduce Corollary~\ref{cor:vcd-acvf} from Theorem~\ref{thm:main}.

\subsection{Proof of Theorem \ref{thm:main}}
\label{subsec:proof-of-main}
In the following, $K$ will be a fixed 
algebraically closed
non-archimedean (complete real-valued) field and $V$ is an affine variety over $K$. We shall freely use the results of Hrushovski and Loeser \cite{HL}  on the spaces $B_{\bF}(X)$ associated to definable subsets $X \subset V$. For the reader's convenience, an exposition (with references) of the results we require below is provided in  \S \ref{subsec:HL}. We shall also make use of some standard facts about singular cohomology of topological spaces; we refer the reader to \S \ref{subsec:cohomology} for a review of these facts. \\

\begin{notation}(closed cube)
\label{not-closed-cube}
For
$R \in \mathbb{R}, R > 0$, and $N >0$,
we denote by $\Cube_N(R)$ the semi-algebraic subset
$\widetilde{\Reali}(\psi,  \bbA^{N}_K)$, where 
\[
\psi = \bigwedge_{1 \leq i \leq N}  |X_i| \leq  R,
\]
and $\bbA^{N}_{K} = \Spec(K[X_1,\cdots,X_{N}])$ is usual affine space. 
Notice that $\Cube_N(R)$ is a closed topological space since the $|X_i|$ are continuous functions (see \ref{enum:BF_prop}(\ref{enum:BF_prop4}), \ref{enum:BF_prop}(\ref{enum:BF_prop7})). Moreover, it is a compact topological space (see \ref{enum:BF_prop}(\ref{enum:BF_prop8})). If $V = \Spec (A) \subset \bbA^{N}_K$ is a closed subvariety, then we set 
$\Cube_V(R):= \Cube_N(R) \cap B_{\bF}(V)$. Note that this a closed semi-algebraic subset of $B_{\bF}(V)$. 
\end{notation}

\begin{notation}(Open, closed  $(\eps,\eps')$-tubes)
\label{not:tubes}
Suppose $\phi(\cdot)$ is a formula in disjunctive normal form without negations and
with atoms of the form $|F| \leq \lambda \cdot |G|$, with $F,G \in K[X_1,\ldots,X_N]$ and $\lambda \in \bbR_+ := \bbR_{\geq 0}$. We denote by
\[
\phi^{+,o}(\cdot;T,T')
\] 
the 
formula obtained from $\phi$ by replacing each atom 
$|F| \leq  \lambda \cdot  |G|$ with $\lambda, G \neq 0$ by the formula 
\[
(|F| < (\lambda \cdot T) \cdot |G|) \vee ((|F| < T') \wedge (|G| < T'),
\] 
and each atom $|F| \leq \lambda \cdot  |G|$ with $\lambda = 0$ or $G =0$ by the formula 
\[
|F| < T',
\] 
where $T,T'$ are new variables of the value sort.
Similarly, we denote by
\[
\phi^{+,c}(\cdot;T,T')
\] 
the
formula  obtained from $\phi$ by replacing each atom 
$|F| \leq  \lambda \cdot |G|$  by the formula 

\[
(|F| \leq (\lambda \cdot T) \cdot |G|) \vee ((|F| \leq T' ) \wedge (|G| \leq T' ),
\]
if $\lambda, G \neq 0$ and by the formula
\[
|F| \leq T' ,
\]
if $\lambda = 0$ or $G=0$. Here again $T,T'$ are new variables of the value sort.\\

For $\eps > 1, \eps' > 0$, and $V$ a closed subvariety of $\bbA^{N}_K$
we set
\begin{align*}
\Tube^{+,o}_{V,\phi}(\eps,\eps') := \widetilde{\Reali}(\phi^{+,o}(\cdot;\eps,\eps'),V), \\
\Tube^{+,c}_{V,\phi}(\eps,\eps') := \widetilde{\Reali}(\phi^{+,c}(\cdot;\eps,\eps'),V).
\end{align*}

 For each $R > 0$, we set 
 \begin{align}
 \label{eqn:def:Tube+o}
 \Tube^{+,o}_{V,\phi}(\eps,\eps',R) := \Cube_{V}(R) \cap \Tube^{+,o}_{V,\phi}(\eps,\eps'), \\
 \label{eqn:def:Tube+c}
 \Tube^{+,c}_{V,\phi}(\eps,\eps',R) := \Cube_{V}(R) \cap \Tube^{+,c}_{V,\phi}(\eps,\eps').
 \end{align}

We set
\[
\CTube^{-,c}_{V,\phi}(\eps,\eps',R) :=
\Cube_V(R) - \Tube^{+,o}_{V,\phi}(\eps,\eps',R).
\]
Notice that by definition, 
$\Tube^{+,o}_{V,\phi}(\eps,\eps',R)$ 
(resp. $\CTube^{-,c}_{V,\phi}(\eps,\eps',R)$) is an open (resp. closed) subset of $\Cube_V(R)$. Moreover, both of these are semi-algebraic as subsets of $B_{\bF}(V)$. \\

Finally, we set
\[
\Boundary^{0,c}_{V,\phi}(\eps,\eps',R) := \Tube^{+,c}_{V,\phi}(\eps,\eps',R) \cap \CTube^{-,c}_{V,\phi}(\eps,\eps',R).
\]
\end{notation}

\begin{remark}
Note that our notation for the `tubes' above is structured so that a superscript $o$ (resp. $c$) in the notation indicates that the corresponding tube is open (resp. closed). 
\end{remark}

The next proposition is the key ingredient for the proof of Theorem \ref{thm:main}. 

\begin{proposition}
\label{prop:main}
Let $V \subset \bbA_K^N$ and $W \subset \bbA_K^{M}$ be closed affine subvarieties. Let $\phi(\cdot,\cdot)$ be a formula in disjunctive normal form  without negations and with atoms of the form $|F| \leq \lambda \cdot |G|$ where $F,G \in K[X_1,\ldots,X_N,Y_1,\ldots, Y_M]$. For each $\bar{w} \in  
W(K)^n$, $\sigma \in \{0,1\}^n$,
and for all sufficiently large $R>0$ and $\delta,\delta',\eps,\eps' \in \mathbb{R}_+$ satisfying,
$0< \delta -1 \ll \delta' \ll \eps -1 \ll \eps' \ll 1$,
\[
\HH^*(\widetilde{\Reali}(\sigma,\bar{w})) \cong \HH^*(S_{\sigma}(\delta,\delta',\eps,\eps',R)),
\]
where $S_{\sigma}(\delta,\delta'\eps,\eps',R)$ is defined by
\[
 S_{\sigma}(\delta,\delta',\eps,\eps',R) := \bigcap_{\sigma(i) = 1} \Tube^{+,o}_{V,\phi(\cdot,w_i)}(\delta,\delta',R) \cap \bigcap_{\sigma(i) = 0} \CTube^{-,c}_{V,\phi(\cdot,w_i)}(\eps,\eps',R),
 \]
 and $\widetilde{\Reali}(\sigma,\bar{w})$ is as in \eqref{eqn:realisigmabarw}.
 \end{proposition}

The proof of Proposition~\ref{prop:main} will use the following lemma.

\begin{lemma}
\label{main:lem} 
With notation as in Proposition \ref{prop:main}:
\begin{enumerate}[1.]
\item 
\label{itemlabel:main:lem:1}
For every fixed $\delta',\eps,\eps', R \in \mathbb{R}_+$,
there exists $\delta_0 = \delta_0(\delta',\eps,\eps',R)  > 1$ such that for all $1 <  t_1 \leq t_2 \leq  \delta_0$, the inclusion map $S_{\sigma}(t_1,\delta',\eps,\eps',R) \hookrightarrow S_{\sigma}(t_2,\delta',\eps,\eps',R)$ is a homotopy equivalence.

\item
\label{itemlabel:main:lem:2}
For every fixed $\eps,\eps', R \in \mathbb{R}_+$,
there exists  $\delta_0'  = \delta_0'(\eps,\eps',R) > 0$ such that for all $0 <  t_1' \leq t_2' \leq \delta_0'$, the inclusion map 
\[
\bigcap_{t >1}S_{\sigma}(t,t_1',\eps,\eps',R) \hookrightarrow 
\bigcap_{t >1}S_{\sigma}(t,t_2',\eps,\eps',R)
\] 
is a homotopy equivalence. 
 
 \item
\label{itemlabel:main:lem:3}
 Let 
 \[
 S'_{\sigma}(\eps,\eps',R) := \bigcap_{t > 1, t' >0}  S_{\sigma}(t,t',\eps,\eps',R).
 \]
For every fixed $\eps', R \in \mathbb{R}_+$,
there exists $\eps_0 = \eps_0(\eps',R)  > 1$ such that for all $1< s_1\leq s_2 \leq \eps_0$, the natural inclusion 
\[
S'_{\sigma}(s_2,\eps',R) \hookrightarrow S'_{\sigma}(s_1,\eps',R)
\] 
is a homotopy equivalence. 

\item
\label{itemlabel:main:lem:4}
For every fixed $R \in \mathbb{R}_+$,
there exists $\eps_0'  = \eps_0'(R) > 0$ such that for all $0< s_1' \leq s_2' \leq \eps_0'$, the natural inclusion 
\[
\bigcup_{s >1}S'_{\sigma}(s,s_2',R) \hookrightarrow \bigcup_{s > 1}S'_{\sigma}(s,s_1',R)
\] 
is a homotopy equivalence.

\item 
\label{itemlabel:main:lem:5}
The following equality holds:
\begin{equation*}
 \widetilde{\Reali}(\sigma,\bar{w}) \cap \Cube_V(R) = \bigcup_{s>1, s'>0} S'_{\sigma}(s,s',R).
 \end{equation*}
 \item
 \label{itemlabel:main:lem:6}
 There exists $R_0>0$, such that 
for all $R >R_0$, the natural inclusion
\[
\widetilde{\Reali}(\sigma,\bar{w}) \cap \Cube_V(R) \hookrightarrow  \widetilde{\Reali}(\sigma,\bar{w})
\]
 is a homotopy equivalence.
\end{enumerate}
\end{lemma}

\begin{remark}\label{rem:limcolim}
\begin{enumerate}
\item The subsets $S_{\sigma}(t,\delta',\eps,\eps',R)$ form an \emph{increasing} sequence in $t$ i.e.
if $t_1 < t_2$, then $S_{\sigma}(t_1,\delta',\eps,\eps',R) \subset S_{\sigma}(t_2,\delta',\eps,\eps',R)$. The analogous assertion also holds for 
$S_{\sigma}(\delta,t',\eps,\eps',R)$ (with $t'$ replacing $t$).
\item The subsets  $S_{\sigma}(\delta,\delta',s,\eps',R)$ form a \emph{decreasing} sequence in $s$ i.e.
if $s_1 < s_2$, then $S_{\sigma}(\delta,\delta',s_2,\eps',R) \subset S_{\sigma}(\delta,\delta',s_1,\eps',R)$. The analogous assertion also holds for 
$S_{\sigma}(\delta,\delta',\eps,s',R)$.
\item Then sequence of subsets $S_{\sigma}(\delta,\delta',\eps,\eps',R)$ is increasing in $R$.
\end{enumerate}

\end{remark}
\begin{proof}[Proof of Lemma \ref{main:lem}]
We prove each part separately below.
\begin{proof}[Proof of Part \eqref{itemlabel:main:lem:1}]
Let 
\[
S^1_{\sigma}(\delta',\eps,\eps',R) = \bigcup_{t > 1} S_{\sigma}(t,\delta',\eps,\eps',R).
\] 

First observe that $S^1_{\sigma}(\delta',\eps,\eps',R)$ is a semi-algebraic subset of $B_{\bF}(V)$. To see this let 
\[
\Phi_{\sigma,\delta',\eps,\eps'}(\cdot; T):= \bigwedge_{i, \sigma(i) = 1}\phi^{+,o}(\cdot,w_i;T,\delta') \wedge \bigwedge_{i, \sigma(i) = 0} \neg \phi^{+,o}(\cdot,w_i;\eps,\eps') \wedge \bigwedge_{1 \leq i \leq N} (|X_i| \leq R) ,
\]
and let
\[
\Phi^1_{\sigma,\delta',\eps,\eps'}(\cdot) := (\exists T) (T > 1) \wedge \Phi_{\sigma,\delta',\eps,\eps'}(\cdot;T).
\]

By \ref{enum:BF_prop}(\ref{enum:BF_prop9}),
\begin{equation*}
S^1_{\sigma}(\delta',\eps,\eps',R) = \widetilde{\Reali}(\Phi^1_{\sigma,\delta',\eps,\eps'},V).
\end{equation*}
It follows that
$S^1_{\sigma}(\delta',\eps,\eps',R)$ is a semi-algebraic subset of $B_{\bF}(V)$.
Now consider the function
$f:  \Reali(\Phi^1_{\sigma,\delta',\eps,\eps'},V) \rightarrow \bbR_+$ 
defined by 
\[
f(x) := \inf_{\{(x,t) \; \mid \;\Phi_{\sigma,\delta', \eps,\eps'}(x;t) \}} t.
\]
It is clear that $f$ is definable. Note that 
\[
S_{\sigma}(t,\delta',\eps,\eps',R) = 
\widetilde{\Reali}(\Phi^1_{\sigma,\delta', \eps,\eps'} \wedge f \geq t,V).
\]
The claim now follows as a direct consequence of Theorem \ref{prop:sublevel}. \\
\end{proof}

\begin{proof}[Proof of Part~\eqref{itemlabel:main:lem:2}]
Let 
\[
S^2_{\sigma}(\eps,\eps',R)  = \bigcup_{t' > 0} \bigcap_{t > 1} S_{\sigma}(t,t',\eps,\eps',R).
\] 
Then, $S^2_{\sigma}(\eps,\eps',R)$ is a semi-algebraic subset of $B_{\bF}(V)$.
To see this let  
\[
\Phi^2_{\sigma,\eps,\eps'}(\cdot;T') = \bigwedge_{\sigma(i) = 1} \phi^{+,c}(\cdot,w_i;1,T') \wedge \bigwedge_{\sigma(i) = 0} \neg\phi^{+,o}(\cdot,w_i;\eps,\eps') \wedge \bigwedge_{1 \leq i \leq N} (|X_i| \leq R),
\]
and 
\[
\Phi^3_{\sigma,\eps,\eps'}(\cdot) := (\exists T' ) (T' > 0) \wedge \Phi^2_{\sigma,\eps,\eps'}(\cdot;T').
\]

As in the previous part,
\begin{equation*}
S^2_{\sigma}(\delta',\eps,\eps',R) = \widetilde{\Reali}(\Phi^3_{\sigma,\eps,\eps'},V).
\end{equation*}

In particular, 
$S^2_{\sigma}(\delta',\eps,\eps',R)$ is semi-algebraic.

Moreover, 
let $g:  \Reali(\Phi^3_{\sigma,\eps,\eps'},V)  \rightarrow \bbR_+$ be the map
defined by 
\[
g(x) := \inf_{\{(x;t') \; \mid \; \Phi^2_{\sigma,\eps,\eps'}(x;t') \} } t'.
\]
Clearly, $g$ is definable and
\[
S^2_{\sigma}(t',\eps,\eps',R) = 
\widetilde{\Reali}(\Phi^3_{\sigma,\eps,\eps'} \wedge g \geq t',V).
\]
As in the previous part, the result follows from an application of Theorem \ref{prop:sublevel} to the map $g$.
\end{proof}

\begin{proof}[Proof of Part \eqref{itemlabel:main:lem:3}]
First, note that $S'_{\sigma}(\eps,\eps',R)$ is semi-algebraic. 
Moreover, we observe that the union $S^3_\sigma(\eps',R) = \bigcup_{s > 1} S'_{\sigma}(s,\eps',R)$ is also a semi-algebraic subset of $B_{\bF}(V)$. To see this 
let  
\[
\Phi^4_{\sigma,\eps'}(\cdot;S) = \bigwedge_{\sigma(i) = 1} \phi^{+,c}(\cdot,w_i;1,0) \wedge \bigwedge_{\sigma(i) = 0} \neg\phi^{+,o}(\cdot,w_i;S,\eps') \wedge \bigwedge_{1 \leq i \leq N} (|X_i| \leq R) .
\]
and 
\[
\Phi^5_{\sigma,\eps'}(\cdot) := (\exists S ) (S > 1) \wedge \Phi^4_{\sigma,\eps'}(\cdot;S).
\]

Then, 
\begin{equation*}
S^3_{\sigma}(\eps',R) = \widetilde{\Reali}(\Phi^5_{\sigma,\eps'},V).
\end{equation*}
In particular,
$S^3_{\sigma}(\eps',R)$ is semi-algebraic.

Let $h: \Reali(\Phi^5_{\sigma,\eps'},V) \rightarrow \bbR_{+} $ be given by 
\[
h(x) = \sup_{\{(x;s)\;| \Phi^4_{\sigma,\eps'}(x,s) \}} s.
\] 

Clearly, $h$ is definable. Moreover, 
\[
S_{\sigma}'(s,\eps',R) = 
\widetilde{\Reali}(\Phi^5_{\sigma,\eps'} \wedge h \geq s,V). 
\]
Now apply Theorem \ref{prop:sublevel}. 
\end{proof}

\begin{proof}
[Proof of Part \eqref{itemlabel:main:lem:4}]
The proof is similar to that of Part \ref{itemlabel:main:lem:3}. 
\end{proof}

\begin{proof}
[Proof of Part \eqref{itemlabel:main:lem:5}]
This follows from the definition of $S'_\sigma(s,s',R)$.
\end{proof}

\begin{proof}
[Proof of Part \eqref{itemlabel:main:lem:6}]
This part follows immediately from Theorem \ref{prop:sublevel}. 
\end{proof}

This completes the proof of Lemma~\ref{main:lem}.
\end{proof}

We now prove Proposition~\ref{prop:main}.
Since the proof is long and technical, we begin by giving a general outline. 
Because of the nature of
the argument the steps enumerated do not actually occur in the same order as in the list below. 

\begin{enumerate}[Step 1.]
\item
\label{itemlabel:prop:main:proof:step:1}

By Lemma \ref{main:lem} (Part \eqref{itemlabel:main:lem:6}), there exists an $R_0 > 0$ such that for all $R > R_0$ the natural inclusion $$\widetilde{\Reali}(\sigma,\bar{w}) \cap \Cube_V(R) \hookrightarrow \widetilde{\Reali}(\sigma,\bar{w})$$ induces an isomorphism:
$$
  \HH^*(\widetilde{\Reali}(\sigma,\bar{w})) \xrightarrow{\cong} \HH^*(\widetilde{\Reali}(\sigma,\bar{w}) \cap \Cube_V(R)). 
$$
So we fix some $R >0$ large enough and consider only the semi-algebraic set
$\widetilde{\Reali}(\sigma,\bar{w}) \cap \Cube_V(R))$.

\item 
\label{itemlabel:prop:main:proof:step:2}
By Lemma \ref{main:lem} (Part \eqref{itemlabel:main:lem:5}), 
we have natural inclusions
$$  S'_{\sigma}(s,s',R) \hookrightarrow  \bigcup_{s>1,s'>0} S'_{\sigma}(s,s',R)=\widetilde{\Reali}(\sigma,\bar{w}) \cap \Cube_V(R).$$ We shall see in Claim \ref{prop:main:claim:4} below that this induces an isomorphism 
\[
\HH^*(\widetilde{\Reali}(\sigma,\bar{w}))  \cap \Cube_V(R)) \cong  \varprojlim_{s'} \varprojlim_{s} \HH^*(S_{\sigma}'(s,s',R)).
\]

\item
\label{itemlabel:prop:main:proof:step:3}
We shall see in Claim \ref{prop:main:claim:1} below that 
the natural inclusions 
\[
S_{\sigma}'(\eps,\eps',R) 
\hookrightarrow S_{\sigma}(t,t',\eps,\eps',R)
\]
induce an isomorphism 
$$
\varinjlim_{t'} \varinjlim_{t} \HH^*(S_{\sigma}(t,t',\eps,\eps',R)) \cong \HH^*(S_{\sigma}'(\eps,\eps',R)) .
$$

\item
\label{itemlabel:prop:main:proof:step:4}
In order to conclude, we shall show that the direct and inverse limits appearing in 
Step~\ref{itemlabel:prop:main:proof:step:2} (proved in Claim \ref{prop:main:claim:6}) and 
Step~\ref{itemlabel:prop:main:proof:step:3} (proved in Claim \ref{prop:main:claim:3})  `stabilize'. This stabilization will result as a consequence of the homotopy equivalences proved in Lemma \ref{main:lem}, and
is proved in two intermediate steps (Claims \ref{prop:main:claim:4} and \ref{prop:main:claim:5} for Step 2, 
and Claims \ref{prop:main:claim:2} and \ref{prop:main:claim:3} for Step 3).
\end{enumerate}

The proofs involving commutation of the limit (or colimit) functors with
cohomology in Steps~\ref{itemlabel:prop:main:proof:step:2} and \ref{itemlabel:prop:main:proof:step:3} all rely on proving that a certain increasing family of compact subspaces 
$S_\lambda \subset T$, 
of a semi-algebraic set $T$, indexed by a real parameter $\lambda$, 
are cofinal in the family of all compact subspaces of 
$S := \cup_{\lambda} S_{\lambda}$ in $T$ (the families are different for different steps). One then uses Lemma \ref{lem:continuity} to obtain the desired commutation of various limits (or colimits) with cohomology. The proofs of all these cofinality statements
rely on the following basic lemma that we extract out for clarity.

\begin{lemma}
\label{lem:cofinal}
Let $T$ be a compact Hausdorff space, $\Lambda$ a partially ordered set, $(C_{\lambda})_{\lambda \in \Lambda}$ an increasing sequence of compact subsets of $T$, and $S := \cup_{\lambda} C_{\lambda}$. Suppose that there is a continuous function $\theta: S \rightarrow \bbR_{> 0} \cup \{\infty\}$ such that the following property holds:
\begin{equation}
\label{eqn:lem:cofinal:property}
\text{\parbox{.85\textwidth}
{For each $\theta_0 \in \bbR_{>0}$, there exists a $\lambda(\theta_0) \in J$ such that $x \in C_{\lambda(\theta_0)}$ if $\theta(x) \geq \theta_0$.}}
\end{equation}
Then the family $(C_{\lambda})_{\lambda \in \Lambda}$ 
is cofinal in the family of compact subsets of $S$ in $T$.
\end{lemma}
\begin{proof}
Let $C \subset S$ be a compact subset  of $S$ in $T$. We need to show that there is a $\lambda$ such that $C \subset C_{\lambda}$.  Since $C$ is compact, $F|_{C}$ attains its minimum $\theta_0 > 0$ on $C$. Let $\lambda(\theta_0)$ be as in the proposition. Clearly,
\[
x \in C \Rightarrow \theta(x) \geq \theta_0 \Rightarrow  x \in C_{\lambda(\theta_0)}.
\]

It follows that $C \subset C_{\lambda(\theta_0)}$, and so the family 
$(C_{\lambda})_{\lambda \in \Lambda}$ is cofinal in the family of compact subsets of $S$ in $T$.
\end{proof}

\begin{proof}[Proof of Proposition \ref{prop:main}]

\begin{customclaim}{1}
\label{prop:main:claim:1}
 The natural inclusions 
\begin{equation}
\label{eqn:def-of-S-sigma-prime}
S_{\sigma}'(\eps,\eps',R) := 
\bigcap_{t> 1,t'>0} S_{\sigma}(t,t',\eps,\eps',R)
\hookrightarrow S_{\sigma}(t,t',\eps,\eps',R)
\end{equation}
induce an isomorphism 
\begin{equation}
\label{eqn:prop:main:claim:1}
\HH^*(S_{\sigma}'(\eps,\eps',R))  \cong \varinjlim_{t,t'} 
\HH^*(S_{\sigma}(t,t',\eps,\eps',R)).
\end{equation}
As an immediate consequence we also have
\begin{equation}
\label{eqn:prop:main:claim:1'}
\HH^*(S_{\sigma}'(\eps,\eps',R))  \cong \varinjlim_{t'} \varinjlim_{t} 
\HH^*(S_{\sigma}(t,t',\eps,\eps',R)).
\end{equation}
\end{customclaim}
(Here the inductive limit in \eqref{eqn:prop:main:claim:1} is taken over the poset $\bbR_{>1} \times \bbR_{>0}$, partially 
ordered by 
\[
(t_1,t_1') \preceq (t_2,t_2') \mbox{ if and only if }  t_2 \leq t_1 \mbox{ and } t_2' \leq t_1',
\]
and  for $(t_1,t_1') \preceq (t_2,t_2')$, 
the morphism
\[
\HH^*(S_{\sigma}(t_1,t'_1,\eps,\eps',R)) \rightarrow
\HH^*(S_{\sigma}(t_2,t'_2,\eps,\eps',R))
\]
is induced from the inclusion
$S_{\sigma}(t_2,t'_2,\eps,\eps',R) \hookrightarrow S_{\sigma}(t_1,t'_1,\eps,\eps',R)$.)

\begin{proof}[Proof of  Claim~\ref{prop:main:claim:1}]
First note that the isomorphism \eqref{eqn:prop:main:claim:1'} is an immediate consequence of
the isomorphism \eqref{eqn:prop:main:claim:1}, and the fact that 
\[
\varinjlim_{t'} \varinjlim_{t} 
\HH^*(S_{\sigma}(t,t',\eps,\eps',R))
\cong
\varinjlim_{t,t'} 
\HH^*(S_{\sigma}(t,t',\eps,\eps',R)).
\]
(see for example   \cite[Expose 1, page 13]{SGA4-tome1} for the last isomorphism).

We now proceed to prove the isomorphism \eqref{eqn:prop:main:claim:1}.
Let 
\[
T =  \bigcap_{i, \sigma(i) = 0} \CTube^{-,c}_{V,\phi(\cdot,w_i)}(\eps,\eps',R).
\]

Since each $\CTube^{-,c}_{V,\phi(\cdot,w_i)}(\eps,\eps',R)$ 
is compact, $T$ is a compact Hausdorff space. Notice that  for
each $t>1,t' > 0$,  $S_{\sigma}(t,t',\eps,\eps',R) \subset T$.

We will now show 
for fixed  $\eps,\eps',R$, 
the family of semi-algebraic sets 
\begin{equation}
\label{eqn:prop:main:claim1:cofinal}
\left(S_{\sigma}(t,t',\eps,\eps',R) \right)_{t > 1, t' >0}
\end{equation}
is a cofinal system of open neighborhoods of 
\[
\bigcap_{t >1,t' > 0}S_{\sigma}(t,t',\eps,\eps',R) 
\]
 in $T$.
Assuming this fact, the claim follows from Part \eqref{itemlabel:lem:continuity:1} of Lemma \ref{lem:continuity}.\\

In order to prove the cofinality statement for the family \eqref{eqn:prop:main:claim1:cofinal}, we first prove the following cofinality statement from which the cofinality of 
\eqref{eqn:prop:main:claim1:cofinal}
will follow.\\

Suppose that $I$ is a finite set, 
and let for each $i \in I$, 
$F_i,G_i \in K[X_1,\ldots,X_N]$, and $\lambda_i \in \bbR_+$.
Let $V$ be as before, $R>0$,
$T^{(1)}$ a compact semi-algebraic subset of $\Cube_V(R)$.
We define 
\begin{equation*}
 S^{(1)}(t,t',R) := T^{(1)} \cap \bigcap_{i \in I} \Tube^{+,o}_{V, |F_i| \leq \lambda_i \cdot  |G_i|}(t,t',R).
\end{equation*}

Notice that  for
each $t>1,t' > 0$,  $S^{(1)}(t,t',R) \subset T^{(1)}$, and hence 
\[
\bigcap_{t >1,t' > 0}S^{(1)}(t,t',R)\subset T^{(1)}
\] 
as well.

\begin{customclaim}{1a}
\label{prop:main:claim:1a}
The family of semi-algebraic sets 
\[
\left(S^{(1)}(t,t',R) \right)_{t > 1, t' >0}
\] 
is a cofinal system of open neighborhoods of 
\[
\bigcap_{t >1,t' > 0}S^{(1)}(t,t',R) 
\]
in $T^{(1)}$.
\end{customclaim}

\begin{proof}[Proof of Claim~\ref{prop:main:claim:1a}]
Proving cofinality of the family 
$
\left(S^{(1)}(t,t',R) \right)_{t > 1, t' >0}
$ 
in the partially ordered family of open neighborhoods of 
\[
\bigcap_{t >1,t' > 0}S^{(1)}(t,t',R) 
\]
is equivalent to proving the cofinality of the family of compact subsets
\[
\left(T^{(1)} - S^{(1)}(t,t',R) \right)_{t > 1, t' >0}
\]
in the partially ordered family of compact subsets of 
$
T^{(1)} - \bigcap_{t >1,t' > 0}S^{(1)}(t,t',R)
$.
For proving the latter we use Lemma~\ref{lem:cofinal}, 
with $\Lambda = \bbR_{>1} \times \bbR_{>0}$, and the 
family $(C_\lambda)_{\lambda \in \Lambda} :=  (T^{(1)} - S^{(1)}(t,t',R))_{(t,t') \in \Lambda}$
of compact semi-algebraic subsets of the compact set $T^{(1)}$.\\

We now define  a  continuous function
$\theta:T^{(1)} - \bigcap_{t >1,t' > 0}S^{(1)}(t,t',R) \rightarrow \bbR_{\geq 0}$.
We first introduce the following auxiliary functions which will be used in the definition of the function $\theta$.
For $\lambda \geq 0$, let $H_\lambda(u,v): \bbR_{\geq 0} \times \bbR_{\geq 0} \rightarrow \bbR_{\geq 0}$ be defined as follows. If $\lambda = 0$, then 
\begin{eqnarray*}
H_0(u,v) &:=& u,
\end{eqnarray*}
and if $\lambda > 0$
\begin{eqnarray}
\label{eqn:prop:main:H}
H_\lambda(u,v) &=& \min(\max(u,v), \max(0,\frac{u}{\lambda v}  -1)),  \mbox{if $v \neq 0$}, \\
&=& u, \mbox{ else}. 
\end{eqnarray}
It is easy to check that the functions $H_\lambda(u,v)$ are continuous. \\
For each $i \in I$, let 
$\theta_i: T^{(1)} - \bigcap_{t >1,t' > 0}S^{(1)}(t,t',R) \rightarrow \bbR_{\geq 0}$ 
be the function defined by
 \[
\theta_i(x) = H_{\lambda_i}(|F_i(x)|,|G_i(x)|),
\]
and let $\theta: T^{(1)} - \bigcap_{t >1,t' > 0}S^{(1)}(t,t',R) \rightarrow \bbR_{\geq 0}$
be defined by
\begin{equation*}
\theta_i(x) = \max_{i \in I} \theta_i(x).
\end{equation*}
Notice that each $\theta_i$, and hence also $\theta$ are continuous, since they are compositions of continuous functions. \\

In order to apply Lemma~\ref{lem:cofinal} it remains to check that $\theta$ is positive, and
that it satisfies \eqref{eqn:lem:cofinal:property} in
Lemma \ref{lem:cofinal}. \\

\begin{enumerate}
\item
$\theta(x) > 0$ for each 
$x \in T^{(1)} - \bigcap_{t >1,t' > 0}S^{(1)}(t,t',R)$: \\
Suppose that $\theta(x) = 0$. This implies that $\theta_i(x) = 0$ for each $i \in I$.

If $\lambda_i=0$, then $\theta_i(x) = 0$ implies that $|F_i(x)| = 0$.
If $\lambda_i > 0$, then $\theta_i(x) = 0$ implies that either $|F_i(x)| = |G_i(x)| = 0$
or $|F_i(x)|/(\lambda_i \cdot |G_i(x)|) \leq 1$ or equivalently $|F_i(x)| \leq \lambda_i \cdot |G_i(x)|$.
Together they imply that $x \in \bigcap_{t >1,t'>0} S^{(1)}(t,t',R)$, 
which is a contradiction. \\

\item
$\theta$ satisfies \eqref{eqn:lem:cofinal:property} in Lemma~\ref{lem:cofinal}, 
with
$\lambda$ defined by $\lambda(\theta_0) = (1+\theta_0,\theta_0)$: \\
Suppose $\theta(x) \geq \theta_0$. First note that
\begin{eqnarray*}
T^{(1)} \setminus S^{(1)}(1+\theta_0,\theta_0,R) &=& T^{(1)} \setminus
\bigcap_{i \in I} \Tube^{+,o}_{V, |F_i| \leq \lambda_i \cdot  |G_i|}(1+\theta_0,\theta_0,R) \\
&=&
T^{(1)} \cap \bigcup_{i \in I} \CTube^{-,c}_{V, |F_i| \leq \lambda_i \cdot  |G_i|}(1+\theta_0,\theta_0,R),
\end{eqnarray*}
which is equal to the set
\[
T^{(1)} \cap \bigcup_{i \in I} 
\widetilde{\Reali}((|F| \geq \lambda_i \cdot (1+\theta_0) \cdot |G|) \wedge ((|F| \geq \theta_0 ) \vee (|G| \geq \theta_0 ))).
\]

Since $\theta(x) \geq \theta_0$, there exists an $i$ such that $\theta(x) = \theta_i(x) = \theta_0$.
This implies that 
$|F_i(x)|$ and $|G_i(x)|$ are not simultaneously $0$. We have two cases.
If $\lambda_i = 0$, then we have that
\[
|F_i(x)|  = \theta_i(x) \geq  \theta_0,
\]
which implies that 
\[
x \in \widetilde{\Reali}(|F_i| \geq (\lambda_i \cdot (1+\theta_0) \cdot |G_i|) \wedge ((|F_i| \geq \theta_0 ) \vee (|G_i| \geq \theta_0).
\]
Otherwise, $\lambda_i > 0$. If $|G_i(x)| \neq 0$,
we have that
\[
\max(|F_i(x)|,|G_i(x)|) \geq \theta_i(x) \geq  \theta_0,
\]
and 
\[
\max(0, \frac{|F_i(x)}{\lambda_i |G_i(x)} - 1) \geq \theta_i(x) \geq \theta_0,
\]
which again implies that 
\[
x \in \widetilde{\Reali}(|F| \geq (\lambda_i \cdot (1+\theta_0) \cdot |G|) \wedge ((|F| \geq \theta_0 ) \vee (|G| \geq \theta_0).
\]
If $|G_i(x)| = 0$, then $|F_i(x)| = \theta_0$, and we have again 
\[
x \in \widetilde{\Reali}(|F| \geq (\lambda_i \cdot (1+\theta_0) \cdot |G|) \wedge ((|F| \geq \theta_0 ) \vee (|G| \geq \theta_0).
\]

\end{enumerate}
This completes the proof 
that
$\theta$ satisfies Property~\eqref{eqn:lem:cofinal:property} in Lemma~\ref{lem:cofinal}
with
$\lambda$ defined by $\lambda(\theta_0) = (1+\theta_0,\theta_0)$, hence
completing the proof of 
Claim~\ref{prop:main:claim:1a}.
\end{proof}

Now we return to the proof the Claim~\ref{prop:main:claim:1}.
Let $\phi = \bigvee_{h \in H} \phi^{(h)}$,
 where each $\phi^{(h)}$ is a conjunction of weak inequalities, 
 $|F_{jh}| \leq \lambda_{jh}\cdot|G_{jh}|$, $j \in J_h$,
 and $H,J_h$ are finite sets.\\
 
 Let $I_\sigma = \{i \in [1,n] \mid \sigma_i =1\}$ and $H^{I_{\sigma}}$ denote the 
 set of maps $\psi: I_{\sigma} \rightarrow H$. 
 Note that 
 \[
  S_{\sigma}(t,t',\eps,\eps',R) = \bigcap_{I_\sigma} \left(\bigcup_{h \in H} \bigcap_{j \in J_h} \Tube^{+,o}_{V,|F_{jh}(\cdot,w_i)| \leq \lambda_{jh}\cdot|G_{jh}(\cdot,w_i)|}(t,t',R)\right) \cap T.
\]
(Recall that 
 \[
 T = \bigcap_{i, \sigma_i = 0} \CTube^{-,c}_{V,\phi(\cdot,w_i)}(t,t',R)
 \]
is a compact semi-algebraic set.)
 Then, 
 \[
 S_{\sigma}(t,t',\eps,\eps',R) = \bigcup_{\psi\in H^{I_{\sigma}}}  S_{\sigma}^{(\psi)}(t,t',\eps,\eps',R),
 \] 
 where  for $\psi \in H^{I_{\sigma}}$
 \[
 S_{\sigma}^{(\psi)}(t,t',\eps,\eps',R) = T\cap \bigcap_{i, \sigma_i = 1} \Tube^{+,o}_{V,\phi^{(\psi(i))}(\cdot,w_i)}(t,t',R).
 \]
 An open neighborhood $U$ of 
 $\bigcap_{t > 1,t' >0} S_{\sigma}(t,t',\eps,\eps',R)$ in $T$ is clearly also an open neighborhood of
 $\bigcap_{t > 1, t' > 0} S_{\sigma}^{(\psi)}(t,t',\eps,\eps',R)$ for each $\psi \in H^{I_{\sigma}}$.
 
  Fixing a $\psi \in H^{I_{\sigma}}$, we apply Claim~\ref{prop:main:claim:1a}, with 
  \begin{eqnarray*}
  T^{(1)} &=& T, \\ 
  I &=& \{(j,\psi(i)) \mid i \in I_\sigma, j \in J_{\psi(i)}\},
  \end{eqnarray*}
  and for $i_0 = (j,\psi(i)) \in I$,
  \begin{eqnarray*}
  F_{i_0} &=& F_{j,\psi(i)}, \\
  G_{i_0} &=& G_{j,\psi(i)},  \\
  \lambda_{i_0} &=& \lambda_{j,\psi(i)}. 
  \end{eqnarray*}
  We obtain that for each $\psi \in H^{I_\sigma}$, there exists $\theta_0^{(\psi)} > 0$, such that 
 \[
 S_{\sigma}^{(\psi)}(1+\theta_0^{(\psi)},\theta_0^{(\psi)},\eps,\eps',R) \subset U. 
 \] 
 
 Now take $\theta_0= \min_{\psi \in H^{I_\sigma}} \theta_0^{(\psi)}$.
 Then, 
 \[
 S_{\sigma}(1+\theta_0,\theta_0,\eps,\eps',R) = \bigcup_{\psi \in H^{I_\sigma}} S_{\sigma}^{(\psi)}(1+\theta_0,\theta_0,\eps,\eps',R) \subset U.
 \] 
 This proves  \eqref{eqn:prop:main:claim:1} and concludes the proof of 
 Claim~\ref{prop:main:claim:1}.
 \end{proof}
 
\begin{customclaim}{2}
\label{prop:main:claim:2}

 The natural inclusions 
\[
\bigcap_{t> 1} S_{\sigma}(t,t',\eps,\eps',R)
\hookrightarrow S_{\sigma}(t,t',\eps,\eps',R)
\]
induce 
for each fixed $t' > 0$, $\eps >1$, $\eps' >0$, $R > 0$,
an isomorphism 
\begin{equation}
\label{eqn:prop:main:claim:2}
\HH^*(\bigcap_{t> 1} S_{\sigma}(t,t',\eps,\eps',R))
  \cong \varinjlim_{t} 
\HH^*(S_{\sigma}(t,t',\eps,\eps',R)).
\end{equation}
\end{customclaim}

\begin{proof}[Proof of  Claim~\ref{prop:main:claim:2}]
The proof is structurally similar to the proof of  Claim~\ref{prop:main:claim:1}.
Let 
\[
T =  \bigcap_{i, \sigma(i) = 0} \CTube^{-,c}_{V,\phi(\cdot,w_i)}(\eps,\eps',R).
\]

Then $T$ is compact. We will now show 
for fixed  $t',\eps,\eps',R$, 
the family of semi-algebraic sets 
\begin{equation}
\label{eqn:prop:main:claim2:cofinal}
\left(S_{\sigma}(t,t',\eps,\eps',R) \right)_{t > 1}
\end{equation}
is a cofinal system of open neighborhoods of 
\[
\bigcap_{t >1}S_{\sigma}(t,t',\eps,\eps',R) 
\]
 in $T$.
Assuming this fact, the claim follows from Part \eqref{itemlabel:lem:continuity:1} of Lemma \ref{lem:continuity}.\\

In order to prove the cofinality statement for the family \eqref{eqn:prop:main:claim2:cofinal},
we first prove the following cofinality statement from which the cofinality of 
\eqref{eqn:prop:main:claim2:cofinal}
will follow.\\

Suppose that $I$ is a finite set, 
and let for each $i \in I$, 
$F_i,G_i \in K[X_1,\ldots,X_N]$, and $\lambda_i \in \bbR_+$. Let $V$ be as before, $R>0$,
and $T^{(2)}$ a compact semi-algebraic subset of $\Cube_V(R)$.
We
define 
\begin{equation*}
 S^{(2)}(t,t',R) := T^{(2)} \cap \bigcap_{i \in I} \Tube^{+,o}_{V, |F_i| \leq \lambda_i \cdot  |G_i|}(t,t',R).
\end{equation*}

\begin{customclaim}{2a}
\label{prop:main:claim:2a}
The family of semi-algebraic sets 
\[
\left(S^{(2)}(t,t',R) \right)_{t > 1}
\] 
is a cofinal system of open neighborhoods of 
\[
\bigcap_{t >1}S^{(2)}(t,t',R) 
\]
in $T^{(2)}$.
\end{customclaim}

\begin{proof}[Proof of Claim~\ref{prop:main:claim:2a}]
To prove that the family of semi-algebraic sets 
\[
\left(S^{(2)}(t,t',R) \right)_{t > 1}
\] 
is a cofinal system of open neighborhoods of 
\[
\bigcap_{t >1}S^{(2)}(t,t',R) 
\]
is equivalent to proving that the family of compact
semi-algebraic sets,
\[
\left(T^{(2)} - S^{(2)}(t,t',R) \right)_{t > 1}
\]  
is cofinal in the family of compact subsets of $T^{(2)}- \bigcap_{t >1}S^{(2)}(t,t',R)$.

Let 
\begin{eqnarray*}
S^{(2)}_i(t,t',R)^c &:=& T^{(2)} \cap \CTube^{-,c}_{V,|F_i| \leq \lambda_i\cdot |G_i|}(t,t',R) \\
&=&
  T^{(2)} \cap \widetilde{\Reali}((|F_{i}| \geq t \cdot \lambda_i \cdot  |G_{i}|) \wedge \\
  && ((|F_{i}|  \geq   t')  \vee (|G_{i}| \geq  t')),V), \mbox{ if $\lambda_i > 0$}, \\
  &=&
  T^{(2)} \cap \widetilde{\Reali}((|F_{i}|  \geq   t'),V), \mbox{ if $\lambda_i = 0$}.
\end{eqnarray*}
Note that 
\[
T^{(2)} - S^{(2)}(t,t',R) =  \bigcup_{i \in I} S^{(2)}_i(t,t',R)^c,
\]
and
\[
T^{(2)}- \bigcap_{t >1}S^{(2)}(t,t',R) =  \bigcup_{i \in I} \bigcup_{t >1}S^{(2)}_i(t,t',R)^c
\]
The last cofinality statement would follow if for each $i$ we can show that 
the family of compact semi-algebraic sets  $\left(S^{(2)}_i(t,t',R)^c \right)_{t > 1}$ is cofinal in the family of compact subspaces of 
$\bigcup_{t >1}S^{(2)}_i(t,t',R)^c$.
This is because if for each compact subspace 
\[
C \subset T^{(2)}- \bigcap_{t >1}S^{(2)}(t,t',R) = \bigcup_{i \in I} \bigcup_{t >1}S^{(2)}_i(t,t',R)^c
\] 
and
$i \in I$,  there exists $t_{0,i} > 1$, such that $C \cap \bigcup_{t >1}S^{(2)}_i(t,t',R)^c \subset S^{(2)}_i(t_{0,i},t',R)^c$, then $C \subset T^{(2)}- S^{(2)}(t_0,t',R)$ with 
$t_0 = \min_i t_{0,i}$. \\

We now proceed to show the cofinality of  the family $\left(S^{(2)}_i(t,t',R)^c \right)_{t > 1}$  in the family of compact subspaces of 
$\bigcup_{t >1}S^{(2)}_i(t,t',R)^c$ using Lemma~\ref{lem:cofinal}.

For each $i \in I$, consider the continuous function
$\theta_{i}:  \bigcup_{t >1}S^{(2)}_i(t,t',R)^c \rightarrow \bbR_+ \cup \{\infty\}$ defined by
\begin{eqnarray}
\nonumber
\theta_i(x) &=& |F_i(x)|
\mbox{ if $\lambda_i = 0$}, \\
\label{eqn:prop:main:claim2a:theta}
\theta_{i}(x) &=&  \frac {|F_{i}(x)|}{\lambda_i|G_{i}(x)|}, \mbox{ if $\lambda_i > 0$}.
\end{eqnarray}

It is an easy exercise to check that the functions $\theta_i$  positive and satisfies Property~
\eqref{eqn:lem:cofinal:property} in Lemma~\ref{lem:cofinal}, 
with the map $\lambda$ defined by 
\begin{eqnarray*}
\lambda(\theta_0) &=& t' \mbox{ if $\lambda_i = 0$}, \\
&=& \theta_0 \mbox{ if $\lambda_i > 0$}.
\end{eqnarray*}
satisfy the hypothesis of 
Lemma~\ref{lem:cofinal}. This finishes the proof of Claim~\ref{prop:main:claim:2a}.
 \end{proof}
 
 The proof of Claim~\ref{prop:main:claim:2} follows from the proof of
 Claim ~\ref{prop:main:claim:2a}, in exactly the same manner as the proof
 of Claim~\ref{prop:main:claim:1} from Claim~\ref{prop:main:claim:1a} and is omitted.
\end{proof}

\begin{customclaim}{3}
\label{prop:main:claim:3}
For every fixed $\eps > 1, \eps' >0$ and $R>0$, there exists $\delta'_0 > 0$ and for each $0< \delta' \leq \delta'_0$, there exists $\delta_0(\delta') > 1$ (depending on $\delta'$) such that  the inclusion 
\[
S_{\sigma}'(\eps,\eps',R) \hookrightarrow S_{\sigma}(\delta,\delta',\eps,\eps',R)
\]
induces an isomorphism
\begin{equation}
\label{eqn:prop:main:claim:3}
\HH^*(S_{\sigma}'(\eps,\eps',R)) \cong \HH^*(S_{\sigma}(\delta,\delta',\eps,\eps',R))
\end{equation}
for all $1 < \delta \leq \delta_0(\delta')$.
\end{customclaim}

\begin{proof}[Proof of Claim \ref{prop:main:claim:3}]
We fix $\eps > 1, \eps' >0 $ and $R>0$. First, note that it follows from \eqref{eqn:prop:main:claim:1'} in Claim \ref{prop:main:claim:1} that
\begin{equation}
\label{eqn:prop:main:claim:3.a}
\HH^*(S_{\sigma}'(\eps,\eps',R)) \cong \varinjlim_{t'} \varinjlim_{t}\HH^*(S_\sigma(t,t',\eps,\eps',R)). 
\end{equation}

By Lemma \ref{main:lem} (Part \eqref{itemlabel:main:lem:2}) 
there exists $\delta_0'$ such that for all $0 < t_2' \leq t_1' \leq \delta_0'$, the inclusion map 
\[
\bigcap_{t >1} S_{\sigma}(t,t_2',\eps,\eps',R) \hookrightarrow \bigcap_{t >1} S_{\sigma}(t,t'_1,\eps,\eps',R)
\] 
induces an isomorphism 
\[
\HH^*(\bigcap_{t > 1} S_{\sigma}(t,t_1',\eps,\eps',R)) \rightarrow \HH^*(\bigcap_{t >1}S_{\sigma}(t,t_2',\eps,\eps',R)).
\]
It follows that, for any $0 < \delta' \leq \delta_0'$.
\begin{equation}
\label{eqn:prop:main:claim:3.b}
 \varinjlim_{t'} \HH^*(\bigcap_{t >1} S_{\sigma}(t,t',\eps,\eps',R)) \cong \HH^*(\bigcap_{t >1}S_{\sigma}(t,\delta',\eps,\eps',R)) 
 \end{equation}

Moreover, it follows from  \eqref{eqn:prop:main:claim:2} that 
\begin{equation}
\label{eqn:prop:main:claim:3.c}
\HH^*(\bigcap_{t >1}S_{\sigma}(t,t',\eps,\eps',R))  \cong
\varinjlim_{t}\HH^*(S_\sigma(t,t',\eps,\eps',R))  
\end{equation}
for each fixed $t' > 0$, $\epsilon > 1$, $\epsilon' > 0$ and $R>0$.
Hence, from \eqref{eqn:prop:main:claim:3.a}, \eqref{eqn:prop:main:claim:3.b}, and \eqref{eqn:prop:main:claim:3.c}
we get an isomorphism

\begin{equation}
\label{eqn:prop:main:claim:3.d}
\HH^*(S_{\sigma}'(\eps,\eps',R)) \cong
\varinjlim_{t}\HH^*(S_\sigma(t,\delta',\eps,\eps',R))  
\end{equation}

It again follows from 
Lemma \ref{main:lem} (Part \eqref{itemlabel:main:lem:1}) that for each fixed $\delta'$,
there exists $\delta_0(\delta')$ such that for all $1 < t_2 \leq t_1 \leq \delta_0(\delta')$ the inclusion map
$S_{\sigma}(t_2,\delta',\eps,\eps',R) \hookrightarrow S_{\sigma}(t_1,\delta',\eps,\eps',R)$ induces an isomorphism 
\[
\HH^*(S_{\sigma}(t_1,\delta',\eps,\eps',R)) \rightarrow \HH^*(S_{\sigma}(t_2,\delta',\eps,\eps',R)),
\]
which implies that 
\begin{equation}
\label{eqn:prop:main:claim:3.e}
 \varinjlim_{t} \HH^*(S_{\sigma}(t,\delta',\eps,\eps',R)) \cong 
 \HH^*(S_{\sigma}(t_0,\delta',\eps,\eps',R))
\end{equation}
for all $0 < t_0 \leq \delta_0(\delta')$.
Claim \ref{prop:main:claim:3} follows from \eqref{eqn:prop:main:claim:3.d} and \eqref{eqn:prop:main:claim:3.e},
after taking $\delta'_0$ and $\delta(\delta')$ as above.
\end{proof}
 
 \begin{customclaim}{4}
 \label{prop:main:claim:4}
 The inclusions 
\[
 \bigcup_{s> 1,s'>0} S_{\sigma}'(s,s',R) \hookrightarrow \widetilde{\Reali}(\sigma,\bar{w}))  \cap \Cube_V(R))
\]
induce an isomorphism 

  \begin{equation}
\label{eqn:prop:main:claim:4}
 \HH^*(\widetilde{\Reali}(\sigma,\bar{w}))  \cap \Cube_V(R)) \cong  \varprojlim_{s',s} \HH^*(S_{\sigma}'(s,s',R)).
 \end{equation}
 As an immediate consequence we also have the isomorphism
   \begin{equation}
\label{eqn:prop:main:claim:4'}
 \HH^*(\widetilde{\Reali}(\sigma,\bar{w}))  \cap \Cube_V(R)) \cong  \varprojlim_{s'}\varprojlim_{s} \HH^*(S_{\sigma}'(s,s',R)).
 \end{equation}
 \end{customclaim}
 
 (Here the projective limit is taken over the poset $\bbR_{>1} \times \bbR_{>0}$, partially 
ordered by 
\[
(s_1,s_1') \preceq (s_2,s_2') \mbox{ if and only if }  s_2 \leq s_1 \mbox{ and } s_2' \leq s_1',
\]
and  for $(s_1,s_1') \preceq (s_2,s_2')$, 
the morphism
\[
\HH^*(S_{\sigma}'(s_2,s'_2,R)) \rightarrow
\HH^*(S_{\sigma}'(s_1,s'_1,R))
\]
is induced from the inclusion
$S_{\sigma}'(s_1,s'_1,R) \hookrightarrow S_{\sigma}'(s_2,s'_2,R)$.)

\begin{proof}[Proof of  Claim~\ref{prop:main:claim:4}]
First note that the isomorphism \eqref{eqn:prop:main:claim:4'} is an immediate consequence of
the isomorphism \eqref{eqn:prop:main:claim:4}, and the fact that 
\[
\varprojlim_{s'} \varprojlim_{s} 
\HH^*(S_{\sigma}'(s,s',R))
\cong
\varprojlim_{s,s'} 
\HH^*(S_{\sigma}'(s,s',R)).
\]
(see for example   \cite[Expose 1, page 13]{SGA4-tome1} for the last isomorphism).
Note that the semi-algebraic sets  $S_{\sigma}'(s,s',R)$ are compact for 
each choice of $s >1, s' > 0$ and $R >0$.
In order to see this, recall that by definition (see \eqref{eqn:def-of-S-sigma-prime}) 
$S_{\sigma}'(s,s',R)$ is the intersection of 
$\bigcap_{i, \sigma(i) = 1} \bigcap_{t> 1,t'>0} \Tube^{+,o}_{V,\phi(\cdot,w_i)}(t,t',R)$,
with the compact semi-algebraic set
$\bigcap_{i, \sigma(i) = 0} \bigcap_{t> 1,t'>0} \CTube^{-,c}_{V,\phi(\cdot,w_i)}(s,s',R)$. Therefore, it suffices to prove that the semi-algebraic set 
\[
\bigcap_{t> 1,t'>0} \Tube^{+,o}_{V,\phi(\cdot,w_i)}(t,t',R)
\] 
is compact
for each $i$. In general, $\phi = \vee_{h \in H} \phi^{(h)}$ where each $\phi^{(h)}$ is a conjunction of weak inequalities $|F_{jh}|<\lambda_{jh}|G_{jh}|$, $j \in J_h$ where $H$ and $J_h$ are finite sets. It follows that the semi-algebraic set $\bigcap_{t> 1,t'>0} \Tube^{+,o}_{V,\phi(\cdot,w_i)}(t,t',R)$ is the union over $H$ of the intersection over $J_h$ of the semi-algebraic sets
$$\bigcap_{t> 1,t'>0} \Tube^{+,o}_{V,|F_{jh}(\cdot,w_i)| \leq \lambda_{jh} \cdot |G_{jh}(\cdot,w_i)|}(t,t',R)$$
We claim that
\begin{equation}
\label{eqn:compact}
\bigcap_{t> 1,t'>0} \Tube^{+,o}_{V,|F_{jh}(\cdot,w_i)| \leq \lambda_{jh} \cdot |G_{jh}(\cdot,w_i)|} = \Cube_V(R) \cap 
\widetilde{\Reali}(|F_{jh}(\cdot,w_i)| \leq \lambda_{jh} \cdot |G_{jh}(\cdot,w_i)|),
\end{equation}
and the latter set is easily seen to be compact.
Verifying the equality in \eqref{eqn:compact} is an easy exercise starting from 
the definition in \eqref{eqn:def:Tube+o}. It follows that
\[
\widetilde{\Reali}(\sigma,\bar{w}))  \cap \Cube_V(R)) = \bigcup_{s>1,s'>0} S_{\sigma}'(s,s',R)
\]
where each $S_{\sigma}'(s,s',R)$ is a compact subset 
of $\widetilde{\Reali}(\sigma,\bar{w}))  \cap \Cube_V(R))$. 
We now prove that 
the family
\begin{equation}
\label{eqn:prop:main:claim4:cofinal}
\left(S_{\sigma}'(s,s',R)\right)_{s >1,s' > 0}
\end{equation}
is cofinal in the family of compact subspaces of 
\[
\widetilde{\Reali}(\sigma,\bar{w}))  \cap \Cube_V(R)) = \bigcup_{s > 1,s' > 0} S_{\sigma}'(s,s',R). 
\]
Then the isomorphism \eqref{eqn:prop:main:claim:4} will follow from
Part~\eqref{itemlabel:lem:continuity:3} of Lemma \ref{lem:continuity}.

In order to prove the cofinality statement for the family \eqref{eqn:prop:main:claim4:cofinal},
we first prove the following cofinality statement from which the cofinality of 
\eqref{eqn:prop:main:claim4:cofinal}
will follow.\\

Suppose that $I$ is a finite set, 
and let for each $i \in I$, 
$F_i,G_i \in K[X_1,\ldots,X_N]$, and $\lambda_i \in \bbR_+$.

Let $V$ and $R>0$ be as before.
We define 
\begin{eqnarray*}
\nonumber
 S^{(3)}(s,s',R) &:=& \bigcup_{i \in I} 
 \CTube^{-,c}_{V, |F_i| \leq \lambda_i \cdot  |G_i|}(s,s',R) \\
 &=& 
 \Cube_V(R) \cap \bigcup_{i \in I} 
 \widetilde{\Reali}(|(F_i| \geq s'),V), \mbox{ if $\lambda_i = 0$},\\
 \nonumber
&=& \Cube_V(R) \cap \bigcup_{i \in I} 
 \widetilde{\Reali}((|F_i| \geq  s\cdot\lambda_i \cdot  |G_i|)  \\
 \nonumber
 &&\wedge (|F_i| \geq s' \vee |G_i| \geq s'),V), 
 \mbox{ if $\lambda_i > 0$}.
\end{eqnarray*}

\begin{customclaim}{4a}
\label{prop:main:claim:4a}
The family of semi-algebraic sets 
\[
\left(S^{(3)}(s,s',R) \right)_{s > 1, s' >0}
\] 
is cofinal in the directed family of 
compact subspaces of 
\[
\bigcup_{s >1,s' > 0}S''(s,s',R).
\]
\end{customclaim}

\begin{proof}[Proof of Claim~\ref{prop:main:claim:4a}]
We use Lemma~\ref{lem:cofinal} using the function 
\[
\theta: \bigcup_{s >1,s' > 0}S^{(3)}(s,s',R) \rightarrow \bbR_{\geq 0} 
\]
defined as follows. For each $i \in I$, let 
$\theta_i: \bigcup_{s >1,s'>0} S^{(3)}(s,s',R)  \rightarrow \bbR_{\geq 0}$ be the function defined by
 \[
\theta_i(x) = H_{\lambda_i}(|F_i(x)|,|G_i(x)|)
\]
(see \eqref{eqn:prop:main:H} to recall definition of $H_{\lambda_i}(\cdot,\cdot)$),
and let
$\theta: \bigcup_{s >1,s'>0} S^{(3)}(s,s',R) \rightarrow \bbR_{\geq 0} $ be defined by
\[
\theta(x) = \max_{i \in I} \theta_i(x).
\]
One can now directly verify that $\theta$ is positive and satisfies
\eqref{eqn:lem:cofinal:property} in Lemma~\ref{lem:cofinal}, 
with the map $\lambda$ defined by 
$\lambda(\theta_0) = (1+\theta_0,\theta_0)$. We leave the details to the reader.
This concludes the proof of Claim~\ref{prop:main:claim:4a}.

\end{proof}

The proof of Claim~\ref{prop:main:claim:4} from Claim~\ref{prop:main:claim:4a}
is formally analogous to the similar derivation of 
Claim~\ref{prop:main:claim:1} from Claim~\ref{prop:main:claim:1a} and is omitted.
\end{proof}

\begin{customclaim}{5}
\label{prop:main:claim:5}

 The natural inclusions 
\[
 S_{\sigma}'(s,s',R) \hookrightarrow
\bigcup_{s > 1} S_{\sigma}'(s,s',R)
\]
induce 
for each fixed $s' > 0$ and  $R > 0$,
an isomorphism 
\begin{equation}
\label{eqn:prop:main:claim:5}
\HH^*(\bigcup_{s > 1} S_{\sigma}'(s,s',R))
  \cong \varprojlim_{s} 
\HH^*(S_{\sigma}'(s,s',R)).
\end{equation}
\end{customclaim}

\begin{proof}[Proof of  Claim~\ref{prop:main:claim:5}]
The proof is structurally similar to the proof of  Claim~\ref{prop:main:claim:4}.

We will now show 
for fixed  $s', R$, 
the family of semi-algebraic sets 
\begin{equation}
\label{eqn:prop:main:claim5:cofinal}
\left(S_{\sigma}(s,s',R) \right)_{s > 1}
\end{equation}
is a cofinal system of compact subsets of 
\[
\bigcap_{s >1}S_{\sigma}(s,s',R).
\]
 in $S$.
Assuming this fact, the claim follows from Part~\eqref{itemlabel:lem:continuity:3} of Lemma~\ref{lem:continuity}.\\

In order to prove the cofinality statement for the family \eqref{eqn:prop:main:claim5:cofinal},
we first prove the following cofinality statement from which the cofinality of 
\eqref{eqn:prop:main:claim5:cofinal}
will follow.\\

Suppose that $I$ is a finite set, 
and let for each $i \in I$, 
$F_i,G_i \in K[X_1,\ldots,X_N]$, and $\lambda_i \in \bbR_+$.
Let $V$ and  $R>0$ be as before.
We
define 
\begin{eqnarray*}
\label{eqn:def-of-S-(4)}
 S^{(4)}(s,s',R) &:=& \bigcup_{i \in I} 
 \CTube^{-,c}_{V, |F_i| \leq \lambda_i \cdot  |G_i|}(s,s',R).
 \nonumber
 \end{eqnarray*}

\begin{customclaim}{5a}
\label{prop:main:claim:5a}
The family of semi-algebraic sets 
\[
\left(S^{(4)}(s,s',R) \right)_{s > 1}
\] 
is a cofinal system of compact semi-algebraic subsets of
\[
\bigcup_{s >1}S^{(4)}(s,s',R). 
\]
\end{customclaim}

\begin{proof}[Proof of Claim~\ref{prop:main:claim:5a}]

Let for each $i \in I$,
\begin{eqnarray*}
S^{(4)}_i(s,s',R) &=&  \CTube^{-,c}_{V, |F_i| \leq \lambda_i|G_i|}(s,s',R) \\
&=&
\Cube_V(R)  \cap \widetilde{\Reali}((|F_{i}|  \geq   s'),V), \mbox{ if $\lambda_i = 0$}, \\
&=&
\Cube_V(R)  \cap \widetilde{\Reali}((|F_{i}| \geq s\cdot \lambda_i \cdot  |G_{i}|)  \\
&& \wedge ((|F_{i}|  \geq   s') \vee (|G_{i}| \geq  s')),V) \mbox{ if $\lambda_i >  0$}.
\end{eqnarray*}

Note that 
\[
S^{(4)}(s,s',R) =  \bigcup_{i \in I} S^{(4)}_i(s,s',R),
\]
and
\[
 \bigcup_{s >1}S^{(4)}(s,s',R) =  \bigcup_{i \in I} \bigcup_{s>1}S^{(4)}_i(s,s',R).
\]
Note that the cofinality statement in our claim would follow if for each $i$ we can show that 
the family of compact semi-algebraic sets  $\left(S^{(4)}_i(s,s',R) \right)_{s > 1}$ is cofinal in the family of compact subspaces of 
$\bigcup_{s >1}S^{(4)}_i(s,s',R)$. To see this, suppose that we have proven the latter cofinality statement (for each $i$). Let $C \subset  \bigcup_{s >1}S^{(4)}(s,s',R) $ be a compact subspace. Then $C_i : = C \cap \bigcup_{s>1}S^{(4)}_i(s,s',R)$ is a compact subspace and by hypothesis
for each $i \in I$,  there exists $s_{0,i} > 1$ such that $C_i \subset S^{(4)}_i(s_{0,i},s',R)$. It follows  that $C \subset S^{(4)}(s_0,s',R)$ with 
$s_0 = \min_i s_{0,i}$. \\

We now proceed to show the cofinality of  the family $\left(S^{(4)}_i(s,s',R) \right)_{s > 1}$  in the family of compact subspaces of 
$\bigcup_{s >1}S^{(4)}_i(s,s',R)$ using Lemma~\ref{lem:cofinal}. For each $i \in I$,
consider the 
continuous
function $\theta_{i}: \bigcup_{s >1}S^{(4)}_i(s,s',R) \rightarrow \bbR_+ \cup \{\infty\}$ defined by

\begin{eqnarray*}
\theta_i(x) &=& |F_i(x)| \mbox{ if $\lambda_i = 0$}, \\
\theta_{i}(x) &=&  \frac {|F_{i}(x)|}{\lambda_i|G_{i}(x)|}, \mbox{ if $\lambda_i > 0$}
\end{eqnarray*}

It is an easy exercise to check that the functions $\theta_i$ are positive and satisfy Property~
\eqref{eqn:lem:cofinal:property} in Lemma~\ref{lem:cofinal}, 
with the map $\lambda$ defined by 
$\lambda(\theta_0) = \theta_0$.
This completes the proof of Claim~\ref{prop:main:claim:5a}.
 \end{proof}

 The proof of Claim~\ref{prop:main:claim:5} follows from the proof of
 Claim ~\ref{prop:main:claim:5a}, in exactly the same manner as the proof
 of Claim~\ref{prop:main:claim:1} from Claim~\ref{prop:main:claim:1a} and is omitted.
\end{proof}

\begin{customclaim}{6}
\label{prop:main:claim:6}
Let  $R>0$. Then there exists $\eps'_0(R) > 0$ (depending on $R$), and for each $0 < \eps' \leq \eps'_0(R)$, there exists $\eps_0(\eps')>1$ (depending on $\eps'$) such that 
\begin{equation}
\label{eqn:prop:main:claim:6}
\HH^*(\widetilde{\Reali}(\sigma,\bar{w}))  \cap \Cube_V(R))) \cong \HH^*(S'_{\sigma}(\eps,\eps',R))
\end{equation}
for all $ 1 < \eps \leq \eps_0(\eps')$.

\end{customclaim}

\begin{proof}[Proof of Claim \ref{prop:main:claim:6}]
It follows from \eqref{eqn:prop:main:claim:4'} in Claim \ref{prop:main:claim:4} that
\begin{equation}
\label{eqn:prop:main:claim:6.a}
\HH^*(\widetilde{\Reali}(\sigma,\bar{w}))  \cap \Cube_V(R)))\cong \varprojlim_{s'} \varprojlim_{s}\HH^*(S_\sigma'(s,s',R)). 
\end{equation}

It follows from Lemma \ref{main:lem} (Part \eqref{itemlabel:main:lem:2}) that 
there exists $\eps_0'(R)$ such that for all $0 < s_2' \leq s_1' \leq \eps_0'(R)$, the inclusion map 
\[
\bigcup_{s>1} S_{\sigma}'(s,s_1',R) \hookrightarrow \bigcup_{s >1} S_{\sigma}'(s,s'_2,R)
\] 
induces an isomorphism 
\[
\HH^*(\bigcup_{s > 1} S'_{\sigma}(s,s_2',R)) \rightarrow \HH^*(\bigcup_{s >1}S'_{\sigma}(s,s_1',R)).
\]
It follows that 
\begin{equation}
\label{eqn:prop:main:claim:6.b}
 \varprojlim_{s'} \HH^*(\bigcup_{s>1} S_{\sigma}'(s,s',R)) \cong \HH^*(\bigcup_{s >1}S_{\sigma}'(s,\eps',R))
 \end{equation}
 for all $0 < \eps' \leq \eps_0'(R).$

Moreover, it follows from  \eqref{eqn:prop:main:claim:5} that 
\begin{equation}
\label{eqn:prop:main:claim:6.c}
\HH^*(\bigcup_{s>1}S_{\sigma}'(s,\eps',R))  \cong
\varprojlim_{s}\HH^*(S_{\sigma}'(s,\eps',R))  
\end{equation}

Hence, from \eqref{eqn:prop:main:claim:6.a}, \eqref{eqn:prop:main:claim:6.b}, and \eqref{eqn:prop:main:claim:6.c}
we get an isomorphism

\begin{equation}
\label{eqn:prop:main:claim:6.d}
\HH^*(\widetilde{\Reali}(\sigma,\bar{w}))  \cap \Cube_V(R))) \cong
\varprojlim_{s}\HH^*(S'_\sigma(s,\eps',R))  
\end{equation}

It again follows from 
Lemma \ref{main:lem} (Part \eqref{itemlabel:main:lem:1}) that for each fixed $s'$, and hence for 
$s' = \eps'$,
there exists $\eps_0(\eps') > 1$ 
such that for all $1 < s_2 \leq s_1 \leq \eps_0(\eps')$, the inclusion map
$S_{\sigma}'(s_1,\eps',R) \hookrightarrow S_{\sigma}'(s_2,\eps',R)$ induces an isomorphism 
\[
\HH^*(S'_{\sigma}(s_2,\eps',R)) \rightarrow \HH^*(S'_{\sigma}(s_1,\eps',R)),
\]
which implies that 
\begin{equation}
\label{eqn:prop:main:claim:6.e}
 \varprojlim_{s} \HH^*(S_{\sigma}'(s,\eps',R)) \cong 
 \HH^*(S_{\sigma}'(\eps,\eps',R)).
\end{equation}
for all $1 < \eps \leq \eps_0(\eps').$
Claim \ref{prop:main:claim:6} follows from \eqref{eqn:prop:main:claim:6.d} and \eqref{eqn:prop:main:claim:6.e}.
\end{proof}

We now return to the proof of Proposition~\ref{prop:main}. Using Lemma \ref{main:lem} (Part \eqref{itemlabel:main:lem:6}), we have that for large enough $R>0$, one has
 \begin{equation}
\label{eqn:prop:main:claim:7}
   \HH^*(\widetilde{\Reali}(\sigma,\bar{w}) \cap \Cube_V(R))  \cong  \HH^*(\widetilde{\Reali}(\sigma,\bar{w})).
  \end{equation}
  It follows from \eqref{eqn:prop:main:claim:7} that 
there exists $R_0 > 0$, such that for all $R \geq R_0$, 
 $\HH^*(\widetilde{\Reali}(\sigma,\bar{w}) \cap \Cube_V(R))  \cong  \HH^*(\widetilde{\Reali}(\sigma,\bar{w}))$.
 Fix $R \geq R_0$.  
 It follows from 
 \eqref{eqn:prop:main:claim:6}
 that there exists $\eps'_0(R) > 0$, and for each $0 < \eps' \leq \eps'_0(R)$, 
 there exists $\eps_0(\eps')>1$ (depending on $\eps'$) such that for all
 $1 < \eps \leq \eps_0(\eps')$,
\begin{equation}
\HH^*(\widetilde{\Reali}(\sigma,\bar{w}))  \cap \Cube_V(R))) \cong \HH^*(S_{\sigma}(\eps,\eps',R)).
\end{equation}
Fix $\eps'$ and $\eps$, satisfying  $0 < \eps' \leq \eps'_0(R)$, and $1 < \eps \leq \eps_0(\eps')$.

Now it follows from \eqref{eqn:prop:main:claim:3} that there exists $\delta'_0(\eps,\eps',R) > 0$ and for each $0< \delta' \leq \delta'_0(\eps,\eps',R)$, there exists $\delta_0(\delta') > 1$ (depending on $\delta'$) such that 
for all $1 < \delta \leq \delta_0(\delta')$,
\begin{equation*}
\HH^*(S_{\sigma}'(\eps,\eps',R)) \cong \HH^*(S_{\sigma}(\delta,\delta',\eps,\eps',R)).
\end{equation*}
Choose $\delta',\delta$ satisfying 
$0< \delta' \leq \delta'_0(\eps,\eps',R)$ and $1 < \delta \leq \delta_0(\delta')$.
It is now clear that with the above choices of $R,\eps',\eps,\delta',\delta$, we have that
\[
\HH^*(\widetilde{\Reali}(\sigma,\bar{w})) \cong \HH^*(S_{\sigma}(\delta,\delta',\eps,\eps',R)).
\]
This concludes the proof of Proposition~\ref{prop:main}.
\end{proof}

We introduce some notation before stating the next Proposition.
For $\delta,\epsilon > 1$ and $ \delta',\eps' > 0$ let 
\[
S''_{\sigma}(\delta,\delta',\eps,\eps',R)= 
\bigcap_{i,\sigma(i)=1} 
\Tube^{+,o}_{V,\phi(\cdot,w_i)}(\delta,\delta',R)
- \bigcup_{i, \sigma(i) = 0}
\Tube^{+,c}_{V,\phi(\cdot,w_i)}(\eps,\eps'). 
\]

Notice that it follows from the above definition that for all 
$\delta,\epsilon > 1$ and $ \delta',\eps' > 0$,
\[
S''_{\sigma}(\delta,\delta',\eps,\eps',R)\subset S_{\sigma}(\delta,\delta',\eps,\eps',R).
\]

Note that 
that the sets $S''_{\sigma}(\delta,\delta',\eps,\eps',R)$ and $S_{\sigma}(\delta,\delta',\eps,\eps',R)$ shrink as $\delta,\delta'$ decreases, and they grow with 
decreasing
$\eps,\eps'$. More precisely, for all $\delta_i,\delta_i',\eps_i,\eps_i', i=1,2$ satisfying
$1 < \delta_1 < \delta_2, 0 < \delta_1' < \delta_2', 1 < \eps_2 < \eps_1, 0 < \eps_2' < \eps_1'$, we have the inclusions
$$\displaylines{
	S_{\sigma}(\delta_1,\delta_1',\eps_1,\eps_1',R) \subset S_{\sigma}(\delta_2,\delta_2',\eps_2,\eps_2',R), \cr
	S''_{\sigma}(\delta_1,\delta_1',\eps_1,\eps_1',R) \subset S''_{\sigma}(\delta_2,\delta_2',\eps_2,\eps_2',R).
}
$$

\begin{proposition}
	\label{prop:connected-comp}
With notation as above, for all $\delta,\delta',\eps,\eps' \in \mathbb{R}_+$ satisfying $0<\delta -1 < \delta' < \eps -1 < \eps'$, every connected component of $S''_{\sigma}(\delta,\delta',\eps,\eps',R)$ is a connected component of the semi-algebraic set 
	\begin{eqnarray}
	\label{eqn:U}
	U_{\phi,\delta,\delta',\eps,\eps',R} &:=&  \bigcap_{1 \leq i \leq n}  (U_{i,\eps,\eps',R} \cap U_{i,\delta,\delta',R}),
	\end{eqnarray}
	where for $1 \leq i \leq n$,
	and $t>1, t' > 0$, $$U_{i,t,t',R} :=  \Cube_V(R) \setminus  \Boundary^{0,c}_{V,\phi(\cdot,w_i)}(t,t',R).$$ 
\end{proposition}

Before proving Proposition~\ref{prop:connected-comp}, we note that Proposition~\ref{prop:connected-comp} and Proposition ~\ref{prop:main} implies:

\begin{proposition}
	\label{prop:reduce-to-open}
	For each 
	$\bar{w} \in W(K)^n 
	$, 
	there exists $\delta > 1,\delta' > 0,\eps >1, \eps'>0$, and $R >0$ such that
	for each
	$\sigma \in \{0,1\}^n$ and $0\leq i < k$,
	one has 
	\begin{equation}
	\label{eqn:prop:reduce-to-open}
	\sum_{\sigma \in \{0,1\}^n} b_i(\widetilde{\Reali}(\sigma,\bar{w})) \leq b_i( U_{\phi,\delta,\delta',\eps,\eps',R}).
	\end{equation}
\end{proposition}
\begin{proof}
By Proposition \ref{prop:main} and using the same notation as in the proof of Proposition \ref{prop:main}, we have that there exist an $R>0$, an $\eps'(R) > 0$ (depending on $R$), and for each $0 < \eps' < \eps_0'(R)$, there exists an $\eps_0(\eps') >$ such that 
\begin{equation}
\HH^*(\widetilde{\Reali}(\sigma,\bar{w}))  \cap \Cube_V(R))) \cong \HH^*(S'_{\sigma}(\eps,\eps',R)).
\end{equation}
 for all $1 < \eps \leq \eps_0(\eps')$. Fix $\eps_i'$ and $\eps_i$ ($i =1,2$), satisfying  $0 < \eps'_1 < \eps_2' \leq \eps'_0(R)$, and $1 < \eps_1 < \eps_2 \leq \eps_0(\eps')$. Now recall that it follows from \eqref{eqn:prop:main:claim:3} that there exists $\delta'_0(\eps_i,\eps_i',R) > 0$ and for each $0< \delta' \leq \delta'_0(\eps_i,\eps_i',R)$, there exists $\delta^{(i)}_0(\delta') > 1$ (depending on $\delta'$ and $\delta'_0(\eps_i,\eps_i',R)$)  such that 
for all $1 < \delta \leq \delta_0^{(i)}(\delta')$,
\begin{equation*}
\HH^*(S_{\sigma}'(\eps_i,\eps_i',R)) \cong \HH^*(S_{\sigma}(\delta,\delta',\eps_i,\eps_i',R)).
\end{equation*}
Let $\delta'$ be such that 
\[
0< \delta' \leq \min(\delta'_0(\eps_1,\eps_1',R),\delta'_0(\eps_1,\eps_1',R))
\] 
and 
\[
1 < \delta \leq \min(\delta^{(1)}_0(\delta'),\delta^{(2)}_0(\delta')).
\]
With the above choices of $R,\eps_i',\eps_i,\delta',\delta$, we have
\[
\HH^*(\widetilde{\Reali}(\sigma,\bar{w})) \cong \HH^*(S_{\sigma}(\delta,\delta',\eps_i,\eps_i',R)).
\]
On the other hand, let $T_i = S_{\sigma}(\delta,\delta',\eps_i,\eps_i',R)$ and $T_i'' = S''_{\sigma}(\delta,\delta',\eps_i,\eps_i',R)$ . Then $T_2 \subset T_1'' \subset T_1$, and the by the previous remarks the natural map
\[
\rH^{i}(T_1) \rightarrow \rH^{i}(T_2)
\]
is an isomorphism. On the other hand, this map factors through $\rH^{i}(T_1'')$ and therefore the natural map
\[
\rH^{i}(T_1'') \rightarrow \rH^{i}(T_1)
\]
is surjective. It follows that $b_i(T_1) \leq b_i(T_1'')$. Since the connected components of the $T_1''$ (as $\sigma$ varies) are connected components of $U_{\phi,\delta,\delta',\eps,\eps',R}$ (by Proposition \ref{prop:connected-comp}), the inequality 
\eqref{eqn:prop:reduce-to-open} follows immediately.
\end{proof}

\begin{proof}[Proof of Proposition \ref{prop:connected-comp}]
Without any loss of generality 
we will assume that
$\phi$ is a 
disjunction of the formulas $\phi_h, h \in H$, where $H$ is a finite set, and each $\phi_h$ is a 
conjunction of  weak inequalities $|F_{hj}| \leq \lambda_{hj} |G_{hj}|, j \in J_h$, 
where 
$J_h$ is a finite set. As before for each $i$ we let $F_{ihj} : = F_{hj}(\cdot,w_i)$, $G_{ihj}:= G_{hj}(\cdot,w_i)$.  \\

We first observe that $S''_{\sigma}(\delta,\delta',\eps,\eps',R) \subset U_{\phi,\delta,\delta',\eps,\eps',R}$.
To see this, for $t'>0, t > 1$,  and $i \in [1,n]$, let $\theta_{i,t,t'}: B_{\bF}(V) \rightarrow \bbR$ be the continuous function defined by 

\begin{eqnarray}
\label{eqn:proof:prop:connected-comp:theta:0}
\theta_{i,t,t'}(x) &=& \max_{h \in H} \min_{j \in J_h} 
\mu_{i,h,j,t,t'}(x),
\end{eqnarray}
where 
\begin{eqnarray*}
\mu_{i,h,j,t,t'}(x)& = & t' - |F_{ihj}(x)|, \mbox{ if $\lambda_{hj} = 0$}, \\
&=& \max(\lambda_j \cdot t \cdot  |G_{ihj}(x)| - |F_{ihj}(x)|,  \\
&&  \hspace{1.2in} \min(t' - |F_{ihj}(x)|, t' - |G_{ihj}(x)|)), \mbox{ if $\lambda_{h j} > 0$}.
\end{eqnarray*}
The formula defining $\theta_{i,t,t'}$ might seem a little formidable at first glance, but becomes
easier to understand with the observation that 
each occurrence of $\max$ and $\min$ in  \eqref{eqn:proof:prop:connected-comp:theta:0} corresponds to 
an occurrence of respectively $\bigvee$ and $\bigwedge$ in the formula $\phi^{+,o}(\cdot;T,T')$ (cf. Notation~\ref{not:tubes}). With this observation,
and the obvious facts that for any  $A \subset \bbR$, 
\begin{eqnarray*}
\bigvee_{a \in A} (a > 0)  &\Leftrightarrow& \max_{a \in A} a  > 0 , \\
\bigwedge_{a \in A} (a > 0)  &\Leftrightarrow&  \min_{a \in A} a  > 0, 
\end{eqnarray*} 
it is easy to verify that
\begin{eqnarray*}
x  \in \Tube^{+,o}_{V,\phi(\cdot,w_i)}(\delta,\delta') \Leftrightarrow
\theta_{i,\delta,\delta'}(x) > 0, \\
x  \in \Tube^{+,c}_{V,\phi(\cdot,w_i)}(\delta,\delta') \Leftrightarrow
\theta_{i,\delta,\delta'}(x) \geq 0,
\end{eqnarray*}
and finally that for any $R > 0$,
\begin{eqnarray}
\label{eqn:proof:prop:connected-comp:theta:1}
&&\\
\nonumber
x  \in \Boundary^{+,c}_{V,\phi(\cdot,w_i)}(\delta,\delta',R) \Leftrightarrow
x \in \Cube_V(R) \wedge (\theta_{i,\delta,\delta'}(x) = 0).
\end{eqnarray}

Now let $x \in S''_{\sigma}(\delta,\delta',\eps,\eps',R)$.
Then, for each $i$ with $\sigma(i) = 1$, 
$x \in \Tube^{+,o}_{V,\phi(\cdot,w_i)}(\delta,\delta',R)$, and hence 
$x \not\in \Boundary^{0,c}_{V,\phi(\cdot,w_i)}(\delta,\delta',R)$. \\

One can also check, using the fact that 
$\delta' < \eps'$ and $\delta < \eps$, that 
$\theta_{i,\delta,\delta'}(x) > 0$ implies that $\theta_{i,\eps,\eps'}(x)  > 0$ as well.
This in turn implies that 
\[
x \in \Tube^{+,o}_{V,\phi(\cdot,w_i)}(\delta,\delta',R) \implies
x \not\in \Boundary^{0,c}_{V,\phi(\cdot,w_i)}(\eps,\eps',R).
\]
Hence, we have that  
\[
x \not\in \Boundary^{0,c}_{V,\phi(\cdot,w_i)}(\delta,\delta',R) \cup \Boundary^{0,c}_{V,\phi(\cdot,w_i)}(\eps,\eps',R)
\]
for all $i$ with $\sigma(i) = 1$. In particular, $x \in U_{i,\eps,\eps',R} \cap U_{i,\delta,\delta',R}$. \\

We now consider the case of all $i$ such that $\sigma(i) = 0$. Suppose that $\sigma(i) = 0$.
Then, 
$x \in \Cube_V(R) - \Tube^{+,c}_{V,\phi(\cdot,w_i)}(\eps,\eps',R)$, and hence 
$x \not\in \Boundary^{0,c}_{V,\phi(\cdot,w_i)}(\eps,\eps',R)$.

Also, if 
$x \not\in \Tube^{+,c}_{V,\phi(\cdot,w_i)}(\eps,\eps',R)$, 
then 
$x \not\in \Boundary^{0,c}_{V,\phi(\cdot,w_i)}(\delta,\delta',R)$, 
since clearly 
\[
\Boundary^{0,c}_{V,\phi(\cdot,w_i)}(\delta,\delta',R) \subset \Tube^{+,c}_{V,\phi(\cdot,w_i)}(\eps,\eps',R),
\]
and hence 
$x \not\in \Boundary^{0,c}_{V,\phi(\cdot,w_i)}(\delta,\delta',R)$ 
either.
Hence, we have that  
\[
x \not\in \Boundary^{0,c}_{V,\phi(\cdot,w_i)}(\delta,\delta',R) \cup \Boundary^{0,c}_{V,\phi(\cdot,w_i)}(\eps,\eps',R)
\]
for all $i$ with $\sigma(i) = 0$. Combining everything, we have $x \in U_{\phi,\delta,\delta',\eps,\eps',R}$.\\

Now let $C$ be a connected component of $S''_{\sigma}(\delta,\delta',\eps,\eps',R)$, and $D$ be the connected component of $U_{\phi,\delta,\delta',\eps,\eps',R}$ containing $C$. We claim that $D = C$. Let $x \in D$, and let $y$ be any point of $C$. Then, since $y \in D$ and $D$ is path connected, there exists a path $\gamma:[0,1] \rightarrow D$, with
$\gamma(0) = y$ and $\gamma(1) = x$, and $\gamma([0,1]) \subset D$. We claim that 
$\gamma([0,1]) \subset S''_{\sigma}(\delta,\delta',\eps,\eps',R)$, which immediately implies that
$D = C$. \\

We first show that for each $i$ with $\sigma(i) = 1$, 
$\gamma([0,1]) \subset  \Tube^{+,o}_{V,\phi(\cdot,w_i)}(\delta,\delta',R)$. 
Consider for each $i$ with $\sigma(i) = 1$,  the continuous function
$\theta_{i}: [0,1] \rightarrow \bbR$ defined by 
\[
\theta_{i}(t) = 
\theta_{i,\delta,\delta'}(\gamma(t)).
\]

Notice that it follows from \eqref{eqn:proof:prop:connected-comp:theta:1} that
$\theta_{i}(t) = 0$ implies that 
\[
\gamma(t) \in  \Boundary^{0,c}_{V,\phi(\cdot,w_i)}(\delta,\delta',R).
\]

Moreover, since 
\[
\gamma([0,1]) \subset \Cube_V(R) \setminus \Boundary^{0,c}_{V,\phi(\cdot,w_i)}(\delta,\delta',R)
\] for each $i$, 
$\theta_i$
cannot vanish anywhere on $[0,1]$. Also notice that $\theta_i(t) > 0$ if and only if 
$\gamma(t) \in  \Tube^{+,o}_{V,\phi(\cdot,w_i)}(\delta,\delta',R)$. 
Since, $\gamma(0) = y  \in S''_{\sigma,\delta,\delta',\eps,\eps',R}$,
this implies that $\theta_i(0) > 0$, and hence $\theta_i(t) > 0$, for each $t \in [0,1]$, and hence
\[
\gamma([0,1]) \subset \bigcap_{i,\sigma(i) = 1} \Tube^{+,o}_{V,\phi}(\delta,\delta',R).
\]

Finally, we show that 
\[
\gamma([0,1]) \subset \bigcap_{i,\sigma(i) = 0} \left(\Cube_V(R) \setminus \Tube^{+,c}_{V,\phi(\cdot,w_i)}(\eps,\eps',R) \right).
\]

Consider for each $i$ with $\sigma(i) = 1$,  the continuous function
$\mu_{i}: [0,1] \rightarrow \bbR$ defined by 
\[
\mu_{i}(t) = 
- \theta_{i,\eps,\eps'}(\gamma(t)).
\]

Notice that $\mu_{i}(t) = 0$ implies that 
$\gamma(t) \in  \Boundary^{0,c}_{V,\phi(\cdot,w_i)}(\eps,\eps',R)$, 
and hence 
since 
$\gamma([0,1]) \subset \Cube_V(R) \setminus \Boundary^{0,c}_{V,\phi(\cdot,w_i)}(\eps,\eps',R)$ 
for each $i$, 
$\theta_i$
cannot vanish anywhere on $[0,1]$. Moreover, also notice that $\mu_i(t) > 0$ if and only if 
$\gamma(t) \in  \Tube^{+,o}_{V,\phi(\cdot,w_i)}(\eps,\eps',R)$. 
Since, $\gamma(0) = y  \in S''_{\sigma}(\delta,\delta',\eps,\eps',R)$,
this implies that $\mu_i(0) > 0$, and hence $\mu_i(t) > 0$, for each $t \in [0,1]$, and hence
\[
\gamma([0,1]) \subset \bigcap_{i,\sigma(i) = 0} \left(\Cube_V(R)  -  \Tube^{o,+}_{V,\phi(\cdot,w_i)}(\eps,\eps',R) \right).
\]
This proves that $D= C$. 
\end{proof}

Let $ X \subset V$ be a definable subset where $V$ is an affine variety of dimension 
$k$, and $U_1,\ldots,U_n$ open semi-algebraic subsets 
of $B_{\bF}(X)$.
For $J \subset [1,n]$, we denote by $U^J := \bigcup_{j \in J} U_j$ and $U_J := \bigcap_{j \in J} U_j$.
We have the following proposition,
which is very similar to \cite[Proposition 7.33, Part (ii)]{BPRbook2}. 

\begin{proposition}
\label{prop:spectral-inequality}
With notation as above, for each $i$, $0 \leq i \leq k$, 
\[
b_i(U_{[1,n]}) \leq \sum_{j=1}^{k-i}\sum_{J \subset [1,n],\card(J) = j} b_{i+j-1}(U^J) + \binom{n}{k-i}b_k(B_{\bF}(V)).
\]
\end{proposition}

\begin{proof}
We first prove the claim when $n = 1$. If $0 \leq i \leq k - 1$, the claim
is
\[ b_i (U_1) \leq 
   b_i (U_1) + b_{k} (B_{\bF}(V)), 
 \]
which is clear. If $i = k$, 
the claim is $b_{k} (U_1) \leq
b_{k} (B_{\bF}(V))$, which is true using Part \eqref{itemlabel:cor:finite-simplicial:d} of 
Corollary \ref{cor:finite-simplicial}.\\

The claim is now proved by induction on $n$. Assume that the induction
hypothesis holds for all $n - 1$ open semi-algebraic subsets of $B_{\bF}(V)$,
and for all $0 \le i \le k$.

It follows from the standard Mayer-Vietoris sequence that
\begin{equation}
  \label{7:eq:propb1} 
b_i (U_{[1,n]}) \leq b_i (U_{[1,n-1]}) + b_i (U_n) + b_{i + 1} (U_{[1,n-1]} \cup U_n) .
\end{equation}

Applying the induction hypothesis to the set $U_{[1,n-1]}$, we deduce that
\begin{eqnarray}
\label{7:eq:propb2}
  b_i (U_{[1,n-1]}) & \leq & \sum_{j = 1}^{k - i}
  \sum_{J \subset [1,n-1], \card(J) = j}  b_{i + j - 1} (U^J)\\
  \nonumber
  & + & \binom{n - 1}{k - i}  b_{k} (B_{\bF}(V)) .
\end{eqnarray}
Next, applying the induction hypothesis to the set,
\[ U_{[1,n-1]} \cup U_n = \bigcap_{1 \leq j \leq n - 1} (U_j \cup U_n), \]
we get that
\begin{eqnarray}
  \label{7:eq:propb3} b_{i + 1} (U_{[1,n-1]} \cup U_n) & \leq &
  \sum_{j = 1}^{k - i - 1} 
  \sum_{
    J \subset [1,n-1], \card(J) = j}
b_{i + j} (U^{J \cup \{n\}})
  \nonumber\\
  & + & \binom{n - 1}{k - i - 1}  b_{k} (B_{\bF}(V)). 
\end{eqnarray}

We obtain from inequalities \eqref{7:eq:propb1}, \eqref{7:eq:propb2}, and \eqref{7:eq:propb3} that

\[ b_i (U_{[1,n]}) \leq \sum_{j = 1}^{k - i}
   \sum_{
     J \subset [1,n], \card(J) = j}
   b_{i + j - 1} (U^J) + \binom{n}{k - i} 
   b_{k} (B_{\bF}(V)),
\]
which finishes the induction.
\end{proof}

\begin{proof}[Proof of Theorem \ref{thm:main}]
Using Proposition \ref{prop:reduce-to-open} we obtain that
for each $i,0 \leq i \leq k$,
\begin{equation}
\label{eqn:proof:thm:main:1}
\sum_{\sigma \in \{0,1\}^n} b_i(\widetilde{\Reali}(\sigma,\bar{w}))  \leq 
b_i( U_{\phi,\delta,\delta',\eps,\eps',R}).
\end{equation}
From the definition of $U_{\phi,\delta,\delta',\eps,\eps',R}$ in \eqref{eqn:U}, we have that 
$U_{\phi,\delta,\delta',\eps,\eps',R}$ is an intersection
of  the sets 
\[
\Cube_V(R) \setminus \Boundary^{0,c}_{V,\phi(\cdot,w_i)}(\eps,\eps',R), 
\]
\[
\Cube_V(R) \setminus \Boundary^{0,c}_{V,\phi(\cdot,w_i)}(\delta,\delta',R), 
\]
for $1 \leq i \leq n$.

Now for each $j \geq 1$ and $J', J'' \subset [1,j]$ with $J' \cup J'' = [1,j], J' \cap J'' = \emptyset$, 
let
\[
\Phi_{J',J''}(\overline{X}, \overline{Y}^{(1)},\ldots,\overline{Y}^{(j)}; s,s',t,t',R)
\] 
denote the conjunction of the following formulas: 
$$\displaylines{
 \bigvee_{j \in J'} \neg (\phi^{+,c}(\overline{X}, \overline{Y}^{(j)};s,s') \vee  \phi^{+,o}(\overline{X}, \overline{Y}^{(j)};s,s')), \cr
\bigvee_{j \in J''} (\phi^{+,c}(\overline{X}, \overline{Y}^{(j)};t,t') \vee \phi^{+,o}(\overline{X}, \overline{Y}^{(j)};t,t')), \cr
\Phi_{V}(\overline{X};R), \cr
\bigvee_{j \in [1,j]}\Phi_{W}(\overline{Y}^{(j)}).
}
$$
where $\Phi_{V,R}(\overline{X};R)$ is a formula such that 
$\Cube_V(R) = \widetilde{\Reali}(\Phi_{V,R})$, and
$\Phi_{W}(\overline{Y})$ is a formula such that 
$B_\bF(W) = \widetilde{\Reali}(\Phi_{W})$.\\

Note that it follows from Notation~\ref{not:tubes}, that for each $j, 1 \leq j \leq n$,
the semi-algebraic set 
\[
\widetilde{\Reali}(
\neg (\phi^{+,c}(\overline{X}, \overline{Y}^{(j)};\cdot,\cdot) \vee  \phi^{+,o}(\overline{X}, w_j;\cdot,\cdot)),V) \cap \Cube_V(R)
\]
is equal to the set 
\[
\Boundary^{0,c}_{V,\phi(\cdot,w_j)}(\cdot,\cdot,R).
\]
It follows that for any  
\[
J' = (j_1',\ldots,j'_{\card(J')}), J'' = (j_1'',\ldots,j''_{\card(J'')}) \subset [1,n]\] 
with $J' \cup J'' = [1,n], J' \cap J'' = \emptyset$,
the semi-algebraic set
\[
\widetilde{\Reali}(\Phi_{J',J''}(\cdot,w_{j_1'},\cdots,w_{j'_{\card(J')}}, w_{j''_1},\cdots, w_{j''_{\card(J'')}}; \eps,\eps',\delta,\delta',R)
\] 
is equal to the
union of the two sets
\[
\bigcup_{j \in J'} (\Cube_V(R) \setminus \Boundary^{0,c}_{V,\phi(\cdot,w_i)}(\eps,\eps',R))
\]
and
\[ 
\bigcup_{j \in J''} (\Cube_V(R) \setminus \Boundary^{0,c}_{V,\phi(\cdot,w_i)}(\delta,\delta',R)).
\]

Denote by $X_{J',J''}$ the definable subset of $V \times W^j \times \bbR^5$ defined by the
formula 
\[
\Phi_{J',J''}(\overline{X}, \overline{Y}^{(1)},\ldots,\overline{Y}^{(j)}; s,s',t,t',R),
\]
and let 
\[
\pi_{j,J',J''}:  X_{J',J''}
\rightarrow W^{J' \cup J''}  \times \bbR^5
\]
  denote the projection. 
It follows from 
Theorem~\ref{prop:finite-homotopy-types} (with  $Y = W^{j}$, $V$ viewed as a quasi-projective variety in $\bbP^{N}$ and $X_{J',J''}$ as above)
that the number of homotopy types amongst
the semi-algebraic sets 
$$B_\bF(\pi_{j,J',J''}^{-1}(w'_1,\ldots,w'_j,s,s',t,t',R))$$ is finite, and moreover since each such fiber is homotopy equivalent to a finite simplicial complex by Theorem~\ref{prop:finite-simplicial},
there exists a finite bound
$C_{i,j,J',J''} \in \bbZ_{\geq 0}$, such that
\[
 b_i(B_\bF(\pi_{j,J',J''}^{-1}(w'_1,\ldots,w'_j,s,s',t,t',R)) \leq C_{i,j,J',J''},
\]
for all  $(w'_1,\ldots,w'_j) \in W(K)^j, s,s',t,t',R \in \bbR$.

Let
\[
C_{i,j} = \max_{J',J'', J'\cup J'' = [1,j],  J'\cap J'' = \emptyset} C_{i,j,J',J''}.
\]
Note that $C_{i,j}$ depend only on $V$ and $\phi$.
Also, since 
each $j$-ary union amongst the 
the semi-algebraic sets 
\[
\Cube_V(R) \setminus \Boundary^{0,c}_{V,\phi(\cdot,w_i)}(\eps,\eps',R),
\]
\[
\Cube_V(R) \setminus \Boundary^{0,c}_{V,\phi(\cdot,w_i)}(\delta,\delta',R), 
\]
is clearly homeomorphic to one of the sets
$B_\bF(\pi_{j,J',J''}^{-1}(w'_1,\ldots,w'_j,s,s',t,t',R))$
the $i$-th Betti number
of every such union is  bounded by $C_{i,j}$.

It now follows from \eqref{eqn:proof:thm:main:1} and Proposition~\ref{prop:spectral-inequality}
that
\[
\sum_{\sigma \in \{0,1\}^n} b_i(\widetilde{\Reali}(\sigma,\bar{w})) \leq 
\sum_{j=1}^{k-i} \binom{2 n}{j} C_{i+j-1,j} + \binom{n}{k-i}b_k(B_{\bF}(V)).
\]
The theorem follows.

\end{proof}

\subsection{Proof of Theorem~\ref{thm:0-1}}
\label{subsec:proof-of-0-1}
We need a couple of preliminary results of a set-theoretic nature starting with the following observation.

\begin{observation}
\label{obs:extension}
Let $Y,Y',V,V',W,W'$ be sets such that $Y \subset V \times W$, $Y' \subset V' \times W'$, $V \subset V'$,
$W \subset W'$,  and $Y' \cap (V\times W)  =Y$. Then,  
for every $n > 0$,
\[
\bigchi_{Y,V,W}(n) \leq  \bigchi_{Y',V',W}(n).
\]
\end{observation}
\begin{proof}
To see this note that a $0/1$ pattern is realized by the tuple $(Y_{w_1},\ldots,Y_{w_n})$ in $V$, only if it is realized by 
the tuple $(Y_{w_1}',\ldots,Y_{w_n}')$ in $V'$. This follows from the fact that $Y' \cap (V\times W)  =Y$, and therefore for all $w \in W$, 
$Y'_w \cap V = Y_w$.
\end{proof}

Let $V,W$ be sets,
$I$ a finite set,  and for each $\alpha \in I$, let $X_\alpha$ be a subset of $V \times W$.
Let $i_\alpha:X_\alpha \hookrightarrow V \times W$ denote the inclusion map.
Suppose that $X$ is a subset of $V \times W$  obtained as a Boolean combination of the $X_\alpha$'s. 
Let $W' = \coprod_{\alpha \in I} W$, and for $\alpha \in I$ we $j_\alpha: W \hookrightarrow W'$ denote the canonical inclusion.
Let
$X' = \bigcup_{\alpha \in I} \mathrm{Im}((1_V \times j_\alpha )\circ i_\alpha) \subset V \times W'$. 
With this notation we have the following proposition.

\begin{proposition}
\label{prop:vcd'}
\[
\bigchi_{X,V,W}(n)  \leq \bigchi_{X',V,W'}( \card(I) \cdot n).
\] 
\end{proposition}

\begin{proof}
For $v \in V$, and $S \subset W$  (resp. $S' \subset W'$)  we set $S_v  :=  S \cap X_v$ (resp. $S'_v := S' \cap X_{v}'$). 
Let $\bar{w} \in W^n$.
We claim that for $v,v' \in V$,  
\begin{align*}
\bigchi_{X,V,W;n}(v,\bar{w}) \neq \bigchi_{X,V,W;n}(v',\bar{w})  \implies &&\\
\bigchi_{X',V,W';\card(I)\cdot n}(v, j_n(\bar{w})) \neq \bigchi_{X',V,W';\card(I)\cdot n}(v', j_n(\bar{w})),&&
\end{align*}
where 
$j_n : W^{[1,n]} \rightarrow W'^{I \times [1,n]}$ is defined by 
\[
j_n(w_1,\ldots,w_n)_{(\alpha,i)} = j_\alpha(w_i).
\] 
To prove the claim first observe that since $\bigchi_{X,V,W;n}(v,\bar{w}) \neq \bigchi_{X,V,W;n}(v',\bar{w})$, there exists $i \in [1,n]$
such that 
$v \in X_{w_i}  \Leftrightarrow  v' \not\in X_{w_i}$.  \\

Since $X$ is a Boolean combination of the 
$X_\alpha, \alpha \in I$, there must exist $\alpha \in I$ such that 
$v  \in (X_\alpha)_{w_i} \Leftrightarrow  v' \not\in (X_\alpha)_{w_i}$.
It now follows from the definition of $X',W'$ that 
$\bigchi_{X',V,W';\card(I) \cdot n}(v,j_n(\bar{w})) \neq \bigchi_{X',V,W';\card(I) \cdot n}(v',j_n(\bar{w}))$.
This implies that 
\[
\card(\bigchi_{X,V,W;n}(V,\bar{w})) \leq \card(\bigchi_{X',V,W';\card(I) \cdot n}(V,j_n(\bar{w}))).
\]
It follows immediately that
\[
\bigchi_{X,V,W}(n)  \leq \bigchi_{X',V,W'}( \card(I) \cdot n).
\] 
\end{proof}

\begin{proof}[Proof of Theorem~\ref{thm:0-1}]
We make two reductions. We first claim that it suffices to prove the theorem in the case of an algebraically closed complete valued field of rank one  
i.e. the value group  subgroup of the multiplicative group $\mathbb{R}_+$. Secondly,
we claim that we can assume without loss of generality that the formula $\phi$ is in disjunctive normal form without negations and
with atoms of the form $|F| \leq \lambda \cdot |G|$.\\

\noindent {\em Reduction to complete algebraically closed field of rank one:}
The theory of algebraically closed valued fields in the two sorted language $\cL$ becomes complete once we fix the
characteristic of the field and that of the residue field. Moreover, for each such characteristic pair $(0,0)$, $(0,p)$, or
$(p,p)$ ($p$ a prime) there exists a model $(K;\Gamma)$ of the theory of algebraically closed valued field such that the value group is a multiplicative subgroup of
$\mathbb{R}_+$ (i.e. of rank one) and $K$ is complete. It follows by a standard transfer argument it suffices to prove the theorem for such a model.
 \\

\noindent
{\em Reduction to the case of disjunctive normal form without negations and with atoms of the form $|F| \leq \lambda \cdot |G|$:}
We now observe that it suffices to prove the theorem in the case when the formula $\phi$ is
equivalent to a formula in disjunctive normal form without negations with atoms of the form $|F| \leq \lambda  \cdot |G|$. 
Furthermore, 
using the first reduction, we may assume that the value group is $\bbR_+$ and $K$ is an algebraically closed complete valued field. In particular,
we assume that the atoms of $\phi$ are of the form $|F| \leq \lambda  \cdot |G|$, with $\lambda \in \bbR_+$, and 
$F,G,\in K[\overline{X},\overline{Y}]$. 
Let $(\phi_\alpha)_{\alpha \in I}$ be the finite tuple of (closed) atomic formulas appearing in $\phi$. Denote by
\[
\phi'' = \left(\bigvee_{\alpha \in I} \left( \phi_\alpha(\overline{X},\overline{Y}^{(\alpha)}) \wedge (|Z_\alpha -1| = 0)\right)\right) \wedge
\bigvee_{\alpha \in I} \theta_\alpha ((Z_\alpha)_{\alpha \in I}),
\]
where $\theta_\alpha((Z_\alpha)_{\alpha \in I})$ is the  closed formula 
\[
(|Z_\alpha - 1| = 0) \wedge \bigwedge_{\beta \neq \alpha}(|Z_\beta| = 0).
\]

Note that  $\phi''$ is
equivalent to a formula in disjunctive normal form without negations and with atoms of the form $|F| \leq \lambda \cdot |G|$. \\

Let $X_{\alpha} :=\Reali(\phi_{\alpha},V\times W)(K)$ and $X = \Reali(\phi,V\times W)(K)$. Then $X$ is a Boolean combination of the $X_{\alpha}$'s and we can define $X' \subset V(K) \times W(K)'$ where 
$X'$ and $W(K)'$ are defined as in Proposition \ref{prop:vcd'}. In particular, we let $\pi_1: X' \rightarrow V(K)$ and $\pi_1': X' \rightarrow W(K)'$ denote the natural projection maps. Similarly, we let
\[
\pi_2'': \Reali(\phi'',V\times W \times \bbA^{|I|})(K) \rightarrow W(K) \times \bbA^{|I|}(K)
\]
and 
\[
\pi_1'': \Reali(\phi'',V\times W \times \bbA^{|I|})(K) \rightarrow V(K) 
\]
denote the natural projection maps. Note that the diagram
\[
\xymatrix{
&\Reali(\phi'',V\times W \times \bbA^{|I|})(K)  \ar[ld]^{\pi_1''} \ar[rd]^{\pi_2''}& \\
V(K) &&  \mathrm{Im}(\pi_2'') 
}
\]
is isomorphic to the diagram
\[
\xymatrix{
&X' \ar[ld]^{\pi_1'} \ar[rd]^{\pi_2'}& \\
V(K) &&  \mathrm{Im}(\pi_2') 
}
\]  
By isomorphism, we mean that there are natural bijections $\Reali(\phi'',V\times W \times \bbA^{|I|})(K) \rightarrow X'$ and $\mathrm{Im}(\pi_2'') \rightarrow \mathrm{Im}(\pi_2')$ making the resulting morphism of diagrams above commute(with identity as the map on $V(K)$). \\

Using Proposition~\ref{prop:vcd'}, we get that
\begin{eqnarray*}
\bigchi_{\Reali(\phi, (V \times W))(K),V(K),W(K)}(n) & \leq&
\bigchi_{X',V(K),(W(K))'}(\card(I) \cdot n),
\end{eqnarray*}
and the right hand side of the above inequality clearly equals
\[
\bigchi_{\Reali(\phi'', (V \times W \times \bbA^{|I|}))(K),V(K),W(K) \times \bbA^{|I|}(K)}(\card(I) \cdot n).
\]
So it suffices to prove that there exists a constant $C$ (depending only on $V$ and $\phi$) such that for all $n$,
\[
\bigchi_{\Reali(\phi'', (V \times W \times \bbA^{|I|}))(K),V(K),W(K) \times \bbA^{|I|}(K)}(n) \leq C \cdot  n^{\dim(V)}.
\]
This shows that we can assume that  $\phi$ is
equivalent to a formula in disjunctive normal form  without negations and with atoms of the form $|F| \leq \lambda \cdot |G|$. \\

We now use the special case of Theorem \ref{thm:main} obtained by setting $i=0$. 
In that case, $b_0(\widetilde{\Reali}(\sigma,\bar{w}))$ is the number of connected components, which is at least one as soon as  $\widetilde{\Reali}(\sigma,\bar{w})$ is non-empty. 
Now use Observation \ref{obs:extension} with $V' = B_{\bF}(V)$, $Y' = \bigcup_{w \in W(K)} \widetilde{\Reali}(\phi(\cdot,w), V \times W)$ and 
$Y = \Reali(\phi,(V \times W))(K)$, noting that there exists a canonical injective map 
$\iota: V(K) \hookrightarrow B_{\bF}(V)$
such that for each $w \in W(K)$ the following diagram of injective maps commutes:
\[
\xymatrix{
V(K)\ar[r]^{\iota_V}   & B_{\bF}(V) \\
\Reali(\phi(\cdot,w),V)(K)\ar[r] \ar[u] & \widetilde{\Reali}(\phi(\cdot,w),V)\ar[u]
}
\]
This finishes the proof.
\end{proof}

\subsection{Proof of Corollary~\ref{cor:vcd-acvf}}
\label{subsec:proof-of-vcd-acvf}
\begin{proof}[Proof of Corollary~\ref{cor:vcd-acvf}]
Corollary~\ref{cor:vcd-acvf} follows immediately from Theorem~\ref{thm:0-1} and the  following
proposition  (Proposition~\ref{prop:vcd}) which is well known, 
but whose proof we include for the sake of completeness.
\end{proof}

\begin{proposition}
\label{prop:vcd}
Suppose that there exists a constant $C>0$ such that for all $n > 0$, 
$ \bigchi_{X,V,W}(n)  \leq C \cdot n^{k}$.
Then,
$\vcd(X,V,W) \leq k$.
\end{proposition}

\begin{proof}
Notice that for $v \in V$ and $w \in W$, $w \in X_v \Leftrightarrow v \in X_w$.
Let $\mathcal{S} = \{X_v \mid v \in V\}$, and
$A = \{w_1,\ldots,w_n\} \subset W$, and $I \subset [1,n]$. For $v \in V$,
$w_i \in X_v$ for all $i \in I$, and $w_i \not\in X_v$ for all $i \in [1,n]\setminus I$ if and only if
$v \in X_{w_i}$ for all $i \in I$, and $v \not\in X_{w_i}$ for all $i \in [1,n]\setminus I$.
This implies that 
\[
\card(\{A \cap Y \mid Y \in \mathcal{S}\}) = \bigchi_{X,V,W;n}(V,\bar{w}) \leq C \cdot  n^k.
\]
The proposition now follows from Definition \ref{def:vcd}.
\end{proof}

\appendix
\section{}
\label{sec:appendix}

\subsection{Review of Singular Cohomology}\label{subsec:cohomology}
In this section we recall some basic statements about singular cohomology groups which are used throughout this article. These facts are all standard and we refer  the reader to \cite{Spanier} for 
their proofs.\\

Given any topological space $X$, one can associate to $X$ the singular cohomology groups $\rH^{i}(X,\bbQ)$ (for $i \geq 0$) which satisfy the following general properties (see for example \cite[page 238-240]{Spanier}):\\

\begin{properties}\ \\
\begin{enumerate}[1.]
\item The $\rH^{i}(X,\bbQ)$ are $\bbQ$-vector spaces. If $X$ is a finite dimensional simplicial complex of dimension $n$, then each $\rH^{i}(X,\bbQ)$ is finite dimensional, and moreover $\rH^{i}(X,\bbQ) = 0$ for all $i > n$. 
\item The singular cohomology groups are contravariant and homotopy invariant i.e. a continuous morphism $f: X \rightarrow Y$ induces a linear map $f^*: \rH^{i}(Y,\bbQ) \rightarrow \rH^{i}(X,\bbQ)$, and if $f$ is a homotopy equivalence, then the induced map $f^{*}$ is an isomorphism. 
\item (Connected components) The dimension of $\rH^{0}(X,\bbQ)$ equals the number of connected components of $X$.
\item For any subspace $Y \subset X$, one can define relative cohomology groups 
\[
\rH^{i}(X,Y; \bbQ)\]
 which fit into a long exact sequence:
\begin{equation*}
\cdots \rightarrow  \rH^{i}(X,Y; \bbQ) \rightarrow \rH^i(X,\bbQ) \rightarrow \rH^i(Y, \bbQ) \rightarrow \rH^{i+1}(X,Y;\bbQ) \rightarrow \cdots
\end{equation*}
\item If $U, V \subset X$ are open subsets such that $U \cup V = X$, then there is a long exact sequence of cohomology groups:
\begin{equation*}
\cdots \rightarrow  \rH^{i}(X, \bbQ) \rightarrow \rH^i(U,\bbQ) \oplus \rH^i(V,\bbQ) \rightarrow \rH^i(U \cap V, \bbQ) \rightarrow \rH^{i+1}(X,\bbQ) \rightarrow \cdots
\end{equation*}
\end{enumerate}
\end{properties}

Finally, we recall some properties of singular cohomology with regards to projective and injective limits. These properties are used in the proof of Proposition \ref{prop:main}. Below, we drop the coefficients $\bbQ$ from the notation of singular cohomology groups. \\

Let $I$ be a directed set, $(U_i)_{i \in I}$ be a directed system of topological spaces, and 
\[
U= \varinjlim_{i} U_i
\] 
denote the corresponding direct limit. In particular, for all $i \leq j $ ($i,j \in I$), we have continuous maps $f_{ij}; U_i \rightarrow U_j$ which induce morphisms $f^*_{ij}: \rH^k(U_j) \rightarrow \rH^k(U_i)$. 
The latter cohomology groups form an inverse system, and the natural continuous maps $U_i \rightarrow U$ induce a morphism 
$$\rH^{k}(U) \rightarrow \varprojlim_i \rH^{k}(U_i).$$ 

Similarly, an inverse system $(U_i)_{i \in I}$ of topological spaces gives rise to a direct system of corresponding cohomology groups and natural morphism 
$$\varinjlim_{i} \rH^{k}(U_i) \rightarrow \rH^{k}(U),$$ 
where 
\[
U = \varprojlim_i U_i.
\]\\

In this article, we only consider direct systems $U_i$ given by an increasing sequences of subspaces of a space $X$ or inverse systems $U_i$ given by a decreasing sequence of subspaces.
In the former case, the direct limit $U$ is given by the union of these spaces, and in the latter case the inverse limit is given by the intersection of these subspaces. The following lemma is our main tool for understanding the corresponding cohomology groups.

\begin{lemma}
\label{lem:continuity}
Let $X$ be a paracompact Hausdorff space
having the homotopy type of a finite simplicial complex, and $I$ a directed set.
\begin{enumerate}[1.]
\item 
\label{itemlabel:lem:continuity:1}
Let $\{U_{i}\}_{i \in I }$ be  
a decreasing
sequence of open subspaces of $X$, and $S : = \bigcap_i U_{i}$. Suppose that the family $U_i$ is cofinal in the family of open neighborhoods of $S$ in $X$. Then the natural map $$\varinjlim_i \rH^{k}(U_i) \rightarrow \rH^k(S)$$ is an isomorphism.

\item 
\label{itemlabel:lem:continuity:3}
Let $\{C_{i}\}_{i \in I}$ be an increasing sequence of 
compact 
subspaces of $S$, and $S : = \bigcup_i C_{i}$. Suppose that the family $C_i$ is cofinal in the family of compact subspaces of $S$. Then the natural map $$\rH^{k}(S) \rightarrow \varprojlim_i \rH^{k}(C_i)$$ is an isomorphism.
\end{enumerate}
\end{lemma}

\begin{proof}[Proof of Part \eqref{itemlabel:lem:continuity:1}]
This is Theorem 5 in \cite{LR}.
\end{proof}
\begin{proof}[Proof of Part \eqref{itemlabel:lem:continuity:3}]
The statement follows from the fact that singular homology of any space is isomorphic to the direct limit of 
the singular homology of its compact subspaces \cite[Theorem 4.4.6]{Spanier}, the fact that 
the singular cohomology group $\HH^*(S,\bbQ)$ is canonically isomorphic to $\Hom(\HH_*(S,\bbQ),\bbQ)$ since $\bbQ$ is a field, and that the dual of a direct limit of finite dimensional vector spaces is the inverse limit of the duals of those vector spaces. 
\end{proof}

\subsection{Recollections from Hrushovski-Loeser}
\label{subsec:HL}
In this section we recall some results from the theory of non-archimedean tame topology due  to Hrushovski and Loeser \cite{HL}. The main reference for this section is Chapter 14 of \cite{HL}, but we refer the reader to \cite{Ducros} for an excellent survey. In particular, we will deal with the model theory of valued fields.  We denote by $K$ a complete valued field with values in the ordered multiplicative group 
 of the positive real numbers. \\
 
We consider a two sorted language with the two sorts corresponding to valued fields and the value group. The signature of this two sorted language will be
\[
(0,1,+_K, \times_K, | \cdot |:K \rightarrow \bbR_+, \leq_{\bbR_+}, \times_{\bbR}),
\]
where the subscript $K$ denotes constants, functions, relations etc., of the field sort and the subscript $\bbR_+$ denotes the same for the value group sort. When the context is clear we will drop the subscripts.\\

We denote by $|\cdot |$ the  valuation written multiplicatively. The valuation $|\cdot|$ satisfies:
\begin{eqnarray*}
|x+y| &\leq& \max\{|x|,|y|\}, \\
|x \cdot y| &=&  |x| |y|, \\
|0| &=& 0.
\end{eqnarray*}

\begin{remark}
Note that we follow Berkovich's convention and write our valuations multiplicatively. In particular, the terminology `valuation' is somewhat abusive, and here we really mean a non-archimedean absolute value. In \cite{HL}, all valuations are written additively. 
\end{remark}

Following \cite[\S 14.1]{HL}, we will denote by $\bF$ the two sorted structure $(K;\bbR_+)$ viewed as a substructure of a model of ACVF. Given a quasi-projective variety $V$ defined over $K$ and an $\bF$-definable subset $X$ of $V \times \bbR^n_+$, Hrushovski and Loeser  \cite{HL} 
associate to $X$ (functorially) a topological space $B_\bF(X)$. By definition, this is the space of types, in $X$, defined over $\bF$ which are almost orthogonal to the definable set $\bbR_+$. Given a variety $V$ as above, we say that subset $Z \subset B_\bF(V)$ is {\it semi-algebraic} if it is of the form $B_\bF(X)$ for an $\bF$-definable subset $X \subset V$. We note that $X$ itself can be identified in $B_\bF(X)$ as the set of simple types, and hence there is a canonically defined injection 
$X \hookrightarrow B_{\bF}(X)$.\\

We now recall a description of the spaces $B_\bF(X)$ in some special cases and some of their properties; these are the only properties which are used in this article.

\begin{properties}\label{enum:BF_prop} \ \\
\begin{enumerate}[1.]
\item\label{enum:BF_prop1}(\cite{HL}, 14.4.1) 
For every $\bF$-definable set $X$, 
$B_\bF(X)$ is a Hausdorff topological space which is locally path connected. This construction is functorial in definable maps i.e. a definable map $f: X \rightarrow Y$ induces a continuous map of the corresponding topological spaces. 

\item\label{enum:BF_prop2}(\cite{HL}, 14.1, pg. 194) If $V$ is an affine variety and $X \subset V$ a definable subset, then $B_\bF(X)$ is a subspace of $B_\bF(V)$.
In fact, it is a semi-algebraic subset (in the sense of Berkovich spaces, see Property \ref{enum:BF_prop3} below).

\item\label{enum:BF_prop3}(\cite{HL}, 14.1, pg. 194)  Suppose $X$ is an affine variety $\Spec(A)$. In this case, $B_\bF(X)$ can be identified with the Berkovich analytic space associated to $X$. Its points can be described in terms of multiplicative semi-norms as follows. A point of $B_\bF(X)$ is a multiplicative map $\phi: A \rightarrow \bbR_+$ such that $\phi(a+b) \leq \max(\phi(a),\phi(b))$. 

\item\label{enum:BF_prop4} With $X= \Spec(A)$, the topology on $B_\bF(X)$ is the one inherited from viewing it as a natural subset of $\bbR^{A}_+$. If $f \in A$, then $f$ gives rise to a {\it continuous} function
	$$|f|: B_\bF(X) \rightarrow \bbR_+$$ defined as follows:
	$$|f|(\phi) = \phi(f) \in \bbR_+.$$ This follows from the previous observation and the definition of the topology on Berkovich analytic spaces. 

\item\label{enum:BF_prop7}((\cite{HL}, 14.1, pg. 194) Let $V = \Spec(A)$. Then any formula $\phi$ of the form $f \bowtie g$, where $f,g \in A$, $\lambda \in \bbR_+$ and $\bowtie \in \{\leq, <,\geq, >\}$ gives a definable subset $X$ of $V$, and therefore a semi-algebraic subset $B_\bF(X)$ of $B_\bF(V)$. It can be described in the language of valuations as the set $\{ x \in B_\bF(V) | f(x) \bowtie \lambda g(x)\}$. In general, the semi-algebraic subset associated to a Boolean combination of such formulas is the corresponding Boolean combination of the semi-algebraic subsets associated to each formula. Moreover, a subset of $B_\bF(V)$ is semi-algebraic if an only if it is a Boolean combination of subsets of the form $\{ x \in B_\bF(X) | f(x) \bowtie \lambda g(x)\}$, where $f,g \in A$, $\lambda \in \bbR_+$ and $\bowtie \in \{\leq, <,\geq, >\}$.

\item\label{enum:BF_prop8}(\cite{HL}, 14.1.2)  If $X$ is an $\bF$-definable subset of an algebraic variety $V$, then $B_\bF(X)$ is compact if and only if $B_\bF(X)$ is closed in $B_F(V')$ where $V'$ is a complete algebraic variety. 

\item\label{enum:BF_prop9} Suppose $V = \Spec(A) \subset \bbA_K^N$ is an affine subvariety, and 
$\phi(X;T)$ (with $X =(X_1,\ldots,X_N)$)
a formula with parameters in $(K;\bbR_+)$. Here $X$ are free variable of the field sort and $T$ is a free variable of the value sort. Suppose $a \in \bbR_+$ such that
for all  $t,t'$ satisfying, $a <t < t'$, $(K;\bbR_+) \models \phi(X;t') \rightarrow \phi(X,t)$. Let
$\psi(X)$ be the formula 
\[
\exists T  (T > a) \wedge \phi(X,T).
\]
Then, 
\[
\widetilde{\Reali}(\psi,V) = \bigcup_{a < t} \widetilde{\Reali}(\phi(\cdot;t),V).
\]

\begin{proof}[Proof of Property~\ref{enum:BF_prop9}]
The inclusion $\bigcup_{a < t} \widetilde{R}(\phi(\cdot;t),V) \subset \widetilde{\Reali}(\psi,V)$ is obvious,
since for each $t > a$, $(K;\mathbb{R}_+) \models \phi(X,t) \rightarrow \psi(X)$, which implies that 
$\widetilde{R}(\phi(\cdot;t),V) \subset \widetilde{R}(\psi(\cdot),V)$.

To prove the reverse inclusion, let $p \in \widetilde{\Reali}(\psi,V)$. Then, by definition $p$ is a type which is 
almost orthogonal to the value group, and  moreover,  there exists $x \in \Reali(\psi,V)(K')$, such that $x \models p$ 
and $(K',\bbR_+)$ is an elementary extension of $(K;\mathbb{R})$. Hence, there exists $t_0 >a$, such that
$(K',\bbR_+) \models \phi(x,t_0)$, and so $p \in \widetilde{\Reali}(\phi(\cdot,t_0),V)$.
This proves that  
\[
\widetilde{\Reali}(\psi,V) \subset \bigcup_{a < t} \widetilde{\Reali}(\phi(\cdot;t),V).
\]
\end{proof}

\end{enumerate}
\end{properties}

Given an $\bF$-definable 
map $f:X \rightarrow \bbR_+$, we will denote by $B_\bF(f): B_\bF(X)\rightarrow B_\bF(\bbR_+) = \bbR_+$ the induced map. We will say that $B_\bF(f)$ is a \emph{semi-algebraic} map.\\

The following theorems which are easily deduced from the main theorems in \cite[Chapter 14]{HL} will play a key role in the results of this paper. We will use the same notation as above.

\begin{theorem-appendix}\cite[Theorem 14.4.4]{HL}
\label{prop:sublevel}
Let $V$ be a quasi-projective variety over $K$,  $X \subset V$ be an $\bF$-definable subset and $f:X \rightarrow \bbR_+$ be an $\bF$-definable map. 
For $t \in \bbR_+$, let $B_\bF(X)_{\geq t}$ denote the semi-algebraic subset 
$B_\bF(X \cap (f \geq t)) = B_\bF(X) \cap (B_\bF(f) \geq t)$ of $B_\bF(V)$.
Then, there exists a finite partition $\cP$ of $\bbR_+$ into intervals, such that for each $I \in \cP$ and 
for all $\eps \leq \eps' \in I$,
the inclusion $B_\bF(X)_{\geq \eps'} \hookrightarrow B_\bF(X)_{\geq \eps}$ is a homotopy equivalence.
\end{theorem-appendix}

\begin{theorem-appendix}\cite[Theorem 14.3.1, Part (1)]{HL}
\label{prop:finite-homotopy-types}
Let $Y$ be a variety and $X \subset Y \times \bbR_+^r \times \bbP^m $ be an $\bF$-definable set. Let $\pi: X \rightarrow Y \times \bbR_+^r $ be the projection map. Then there are finitely many homotopy types amongst the fibers $(B_\bF(\pi^{-1}(y;t)))_{(y;t) \in Y \times \bbR_+^r}$.
\end{theorem-appendix}

\begin{theorem-appendix}\cite[Theorem 14.2.4]{HL}
\label{prop:finite-simplicial}
Let $V$ be a quasi-projective variety defined over $K$, and $X$ an $\bF$-definable subset of $V$. Then there exists a sequence of 
finite simplicial complexes $(X_i)_{i \in \bbN}$ embedded in $B_\bF(X)$ of dimension $\leq \dim(V)$, deformation retractions  $\pi_{i,j}: X_i \rightarrow X_j, j < i$, and deformation retractions $\pi_i: B_\bF(X) \rightarrow X_i$, such that
$\pi_{i,j} \circ \pi_i = \pi_j$ and the canonical map $B_\bF(X) \rightarrow \varprojlim_{i} X_i$ is a homeomorphism.
\end{theorem-appendix}

As an immediate consequence of Theorem \ref{prop:finite-simplicial} we have using the same notation:
\begin{corollary-appendix}
\label{cor:finite-simplicial}
Let $V \subset \bbA_K^N$ be a closed affine subvariety, and let $B_\bF(X)$ be a semi-algebraic subset of $V$.
\begin{enumerate}[(a)]
\item 
\label{itemlabel:cor:finite-simplicial:a}
Every connected component of  $B_\bF(X)$ is path connected.
\item
\label{itemlabel:cor:finite-simplicial:b}
$\HH^i(B_{\bF}(X)) = 0$ for $i > \dim(V)$.
\item
\label{itemlabel:cor:finite-simplicial:c}
$\dim \HH^*(B_\bF(X)) < \infty$.
\item
\label{itemlabel:cor:finite-simplicial:d}
The restriction homomorphism $\HH^{\dim(V)}(B_\bF(V)) \rightarrow \HH^{\dim(V)}(B_\bF(X)) $ is surjective.
\end{enumerate}
\end{corollary-appendix}
Note that Parts \eqref{itemlabel:cor:finite-simplicial:a}, \eqref{itemlabel:cor:finite-simplicial:b} and 
\eqref{itemlabel:cor:finite-simplicial:c}
 follow directly from Theorem~\ref{prop:finite-simplicial}.
 
\begin{proof}[Proof of Part \eqref{itemlabel:cor:finite-simplicial:d}]
Recall the definition of $\Cube_V(R)$ (cf. Notation~\ref{not-closed-cube}) and that 
$\Cube_V(R)$  is a compact topological space. Similar remarks apply to $\Cube_V(R) \cap X$. Moreover, arguing as in Part \eqref{itemlabel:main:lem:6} of Lemma  \ref{main:lem}, for sufficiently large $R$ the natural inclusions $\Cube_V(R) \cap X \hookrightarrow B_{\bF}(X)$ and $\Cube_V(0,R) \hookrightarrow B_{\bF}(V)$ induce homotopy equivalences.
Therefore, it is sufficient to prove that for all sufficiently large $R > 0$ the natural induced morphism
\[
\HH^{\dim(V)}(\Cube_V(R)) \rightarrow \HH^{\dim(V)}(\Cube_V(R) \cap X)
\]
is surjective.\\

By Theorem \ref{prop:finite-simplicial}, $B_\bF(V)$ and hence $\Cube_V(R)$ 
for every $R >0$,
has the homotopy type of a finite simplicial polyhedron of dimension at most $\dim(V)$.
Since $\Cube_V(R)$ is compact, it follows that the cohomological dimension (in the sense of \cite[page 196, Definition 9.4]{Iversen}) of $\Cube_V(R)$ is $\leq \dim(V)$.\\

It follows again from Theorem \ref{prop:finite-simplicial} that there exists a compact polyhedron $Z \subset \Cube_V(R) \cap X$ (for sufficiently large $R>0$) such that $Z$ is a deformation retract of $B_\bF(X)$. Let $\iota: Z \hookrightarrow \Cube_V(R) \cap X$ be the inclusion map. Note that $\iota$ induces isomorphisms in cohomology.
Since the inclusion of $Z$  in $\Cube_V(R)$ factors through $\iota$, 
and $\iota$ induces isomorphisms in cohomology,
it follows (using the long exact sequence of cohomology for pairs) that
\[
\HH^*(\Cube_V(R),\Cube_V(R) \cap X) \cong \HH^*(\Cube_V(R),Z).
\]

We now prove that 
\[
\HH^{\dim(V)+1}(\Cube_V(R),\Cube_V(R) \cap X) \cong \HH^{\dim(V)+1}(\Cube_V(R),Z) = 0.
\]
This gives the desired result by an application of the long exact sequence in cohomology associated to the pair
$(\Cube_V(R),\Cube_V(R) \cap X)$.\\

Recall that $\Cube_V(R)$ is a compact space, and consequently  that $Z$ is a closed
subspace of $\Cube_V(R)$.
It follows now \cite[page 198, Proposition 9.7]{Iversen} that the cohomological dimension of $U := (B_\bF(V) \cap \Cube_V(R)) \setminus Z$ is also  $\leq \dim(V)$.
This implies that $\HH^{\dim(V)+1}_c(U) \cong \HH^{\dim(V)+1}(\Cube_V(R),Z) =0$, which finishes the proof.
\end{proof}

\section*{Acknowledgments} 
S.B.  thanks Institut Henri Poincar\'{e}, Paris,  for hosting him during the trimester
on ``Model Theory, Combinatorics and Valued fields''  where  part of this work was done. Both authors thank the anonymous referees for very detailed reports which helped to substantially improve the paper.

\bibliographystyle{amsalpha}

\bibliography{master}

\end{document}